\documentclass[reqno]{amsart}
\usepackage[left=1in, top=1in, right=1in, bottom=1in]{geometry}
\usepackage{amsmath, amsthm, amssymb}
\usepackage{mathtools}
\usepackage{mathrsfs}
\usepackage{graphicx}
\usepackage{subcaption}
\usepackage[all]{xy}
\usepackage{enumerate}

\usepackage{tikz}
\usetikzlibrary{math,angles,quotes,intersections,calc,through}

\newtheorem{theorem}{Theorem}[section]

\newtheorem{corollary}[theorem]{Corollary}
\newtheorem{proposition}[theorem]{Proposition}

\theoremstyle{definition}
\newtheorem{definition}[theorem]{Definition}
\newtheorem{remark}[theorem]{Remark}

\numberwithin{equation}{section}

\DeclareMathOperator{\Tr}{Tr}

\def\bC{\mathbb{C}}

\def\bR{\mathbb{R}}

\def\cM{\mathcal{M}}
\def\cT{\mathcal{T}}

\def\dsds{\frac{\partial s_1}{\partial s}}
\def\dsdu{\frac{\partial s_1}{\partial u}}
\def\duds{\frac{\partial u_1}{\partial s}}
\def\dudu{\frac{\partial u_1}{\partial u}}

\def\dgs{\dot{\gamma}(s)}
\def\dgso{\dot{\gamma}(s_1)}

\def\dgis{\gamma'_j(s)}
\def\dgjs{\gamma'_k(s)}
\def\dgiss{\gamma'_j(s_1)}
\def\dgjss{\gamma'_k(s_1)}
\def\ddgis{\gamma''_j(s)}
\def\ddgiss{\gamma''_j(s_1)}

\def\ddgs{\gamma''(s)}

\def\Fi{F_{v_j}(v)}
\def\Fij{F_{v_jv_k}(v)}

\def\ta{\tilde{a}}
\def\tb{\tilde{b}}

\def\pa{\partial}

\def\ds{\displaystyle}

\begin{document}

\title{Twist Coefficients of Periodic Orbits of Minkowski Billiards}

\author[C. Villanueva]{Carlos Villanueva$^*$}
\address{Department of Mathematics, University of Oklahoma, Norman, OK 73019, USA}
\email{\tt carlosvch@ou.edu}
\thanks{C.V. is supported in part by the Department of Defense SMART Scholarship OMB NO. 0704-0466.}
\thanks{$^*$ Author to whom any correspondence should be addressed.}

\author[P. Zhang]{Pengfei Zhang}
\address{Department of Mathematics, University of Oklahoma, Norman, OK 73019, USA}
\email{\tt pengfei.zhang@ou.edu}

\subjclass[2020]{37G05 37C83 37E40 53B40}

\keywords{Minkowski norm, Finsler norm,  billiards, normal form,
twist coefficient, nonlinear stability}

\dedicatory{To Leonid Bunimovich on his 75th birthday}

\begin{abstract}
We investigate the fundamental properties of Minkowski billiards and introduce a new coordinate system $(s,u)$ on the phase space $\mathcal{M}$. In this coordinate system, the Minkowski billiard map $\mathcal{T}$ preserves the standard area form $\omega = ds \wedge du$. We then classify the periodic orbits of Minkowski billiards with period $2$ and derive formulas for the twist coefficient $\tau_1$ for elliptic periodic orbits, expressed in terms of the geometric characteristics of the billiard table. Additionally, we analyze the stability properties of these elliptic periodic orbits.
\end{abstract}

\maketitle

\tableofcontents

\section{Introduction}\label{sec.introduction}

Dynamical billiards describe a system in which a particle moves within a container and reflects off its walls according to specific reflection rules.
In \cite{Bi27, Bi27b} Birkhoff studied the dynamical billiards on convex domains with smooth boundaries and proved the existence of periodic orbits of period $n$ for any $n\ge 2$.
In \cite{La73} Lazutkin proved the existence of caustics of such convex billiards near the boundaries of the phase space.
Sinai pioneered the study of chaotic billiards. In his seminal paper \cite{Sin70}, he investigated dynamical billiards with dispersing boundaries and proved both the hyperbolicity and ergodicity of such systems.
In \cite{Bu74a, Bu74}, Bunimovich observed that convex circular arcs could be used to construct chaotic billiards and introduced the stadium billiard, which became a favorite model among physicists. The construction of chaotic billiards was significantly extended by Wojtkowski \cite{Woj86}, Markarian \cite{Mar88}, Donnay \cite{Don91}, and Bunimovich \cite{Bu92}.
Since then, several classes of chaotic billiards have been discovered. Notable examples include the elliptic stadium \cite{MaKaPdC}, track billiards \cite{BuDM}, asymmetric lemon billiards \cite{BZZ, JZ21}. An overview of chaotic billiards can be found in \cite{ChMa}.

Birkhoff Normal Form plays an important role in the study of stability properties of elliptic periodic orbits, see \cite{Bi27, SiMo}. Under certain nonresonance assumptions, the dynamical system around an elliptic fixed point can be viewed as a perturbation of an integrable model, 
which is called the normal form of the map. If the integrable model has nonzero twist, then the elliptic fixed point must be nonlinearly stable \cite{Mos56, Mos73}. That is, the fixed point is contained in a nesting sequence of invariant disks whose boundaries are smooth invariant curves.
Generally speaking, the twist coefficients can be represented as rational functions of the coefficients of the Taylor polynomials at the elliptic fixed point, see \cite{Moe90, CGM}. These rational functions, albeit explicit, are generally very involved.

For dynamical systems with geometric background, it is expected that these twist coefficients can be represented as much simpler rational functions of the geometric quantities such as distances and curvatures. In \cite{KP05}, Kamphorst and Pinto-de-Carvalho showed that the first twist coefficient of a elliptic periodic orbit of period 2 of a planar billiard can be expressed a simple rational function in terms of the geometric quantities such as the orbit length and the radii of curvature of the boundary of the billiard table.
In \cite{JZ22}, Jin and Zhang obtained an explicit formula of the second twist coefficient
of periodic orbits of period 2 in terms of the same geometric quantities of the billiards.
The authors also provided  applications of their formula in the study of stability properties to various billiard systems.

Dynamical billiards can also be defined on compact domains of Riemannian manifolds,
on which a particle moves along geodesics in the interior of the domain and makes elastic reflections upon impact on the boundaries, see \cite{Vet, Zha17}. Therefore, billiard flows can be viewed as extensions of geodesic flows to Riemannian manifolds with boundaries. Interestingly,
dynamical billiards on curved surfaces are related to the study of quantum magnetic confinement of non-planar 2D electron gases (2DEG) in semiconductors \cite{FLBP}, where the effect of varying the curvature of the surface corresponds to a change in the potential energy of the system.

In \cite{GT02}, Gutkin and Tabachnikov expanded the theory by introducing dynamical billiards on Minkowski spaces, and more generally, on Finsler manifolds. Unlike Euclidean and Riemannian settings, these spaces are not equipped with an inner product but rather a norm. As a result, the concept of angles, which plays a crucial role in defining reflections in Euclidean and Riemannian billiards, no longer exists. Instead, in \cite{GT02}, they revisited the fundamental concept of reflections through the critical points of the trajectory length function. Since then, several notable results have emerged regarding Minkowski and Finsler billiards (see \cite{AO14, AFOR}), along with some unexpected applications in symplectic and contact geometry (see \cite{AKO14}).

In this paper, we investigate the properties of periodic orbits in dynamical billiards on Minkowski planes. Let $F:\mathbb{R}^2 \to \mathbb{R}$ be a Minkowski norm and $Q \subset \mathbb{R}^2$ be a connected, bounded domain with piecewise smooth boundaries. We consider a point mass moving freely inside $Q$, with reflections occurring at the boundary according to Finsler reflection laws (see Proposition~\ref{pro.Finsler.reflection.law}), thereby defining a Minkowski dynamical billiard system (or {\it Minkowski billiards} for short).
The phase space $\mathcal{M}$ comprises all unit vectors at the boundary $\partial Q$, pointing inward toward $Q$. Notably, this definition aligns with that of Euclidean billiards. However, the conventional arclength-angle coordinate system $(s, \theta)$ is no longer applicable, as the angle is undefined in a normed space.
We introduce a new coordinate system $(s,u)$ for the phase space $\mathcal{M}$ and demonstrate that the Minkowski billiard map $\mathcal{T}$ is symplectic using the generating function method. Subsequently, we express the tangent map $D\mathcal{T}$ in terms of geometric quantities of the billiard table, see Proposition~\ref{pro.DT.Minkowski}. This enables classification of period-$2$ orbits as elliptic, parabolic, or hyperbolic.
Under mild nonresonance conditions for elliptic orbits, we derive a concise formula for the first twist coefficient $\tau_1$ in terms of the geometric quantities of the billiard table, see Theorem~\ref{thm.t1.sym} and \ref{thm.t1.asym}. Combining with Moser's twist mapping theorem, we investigate specific Minkowski billiard families and determine conditions for nonlinear stability of periodic orbits, see Section~\ref{sec.applications}.

\section{Preliminaries}

\subsection{Minkowski spaces and Finsler manifolds}
Finsler geometry has its origin in the study of the calculus of variations. It allows to study more expansive cases such as objects traveling in anisotropic mediums where speed depends on the direction of travel. For a thorough introduction to Minkowski Spaces and Finsler Geometry, see \cite{BCS, ChSh, Ru59}. 

\begin{definition}\label{def.Minkowski}
A \textit{Minkowski norm} on the space $\bR^n$ is a function 
$F: \bR^n \rightarrow [0, \infty)$ satisfying the following conditions:
\begin{enumerate}[(i)]

		\item Regularity: $F$ is $C^\infty$ on $\bR^n \backslash \{0\}$.
		
		\item Positive homogeneity: $F(c \, v) = c\, F(v)$ for any number $c >0$ and any vector $v \in \bR^n$. 
		
		\item Strong convexity: The $n \times n$ Hessian matrix 
		$ (g_{jk}(v)) := \bigg( \frac{\pa^2}{\pa v^j \pa v^k} \big( \frac{1}{2} F^2\big)(v) \bigg)$
		is positive definite at every point in $\bR^n \backslash \{0\}$.
\end{enumerate}
Given a Minkowski norm $F$ on $\bR^n$, the pair $(\bR^n, F)$ is called a \textit{Minkowski space}. 
The subset $I=\{v\in \bR^n: F(v) =1\}$ of all $F$-unit vectors is called  the \textit{indicatrix} of the Minkowski space. 
\end{definition} 
It is clear that the norm $F$ can be recovered from its indicatrix $I$.

\begin{definition}\label{def.Finsler}
Let $M$ be an $n$-dimensional closed manifold. 
A \textit{Finsler structure} on the manifold $M$ is a function 
$F: TM \rightarrow [0, \infty)$ satisfying the following conditions:
\begin{enumerate}[(i)]

		\item Regularity: $F$ is $C^\infty$ on $TM \backslash \{0\}$.
		
		\item Positive homogeneity: $F(x, c \, v) = c\, F(x, v)$ for any $c >0$ and any $v \in T_xM$. 
		
		\item Strong convexity: The $n \times n$ Hessian matrix 
		$ (g_{jk}(x, v)) := \bigg(\frac{\pa^2}{\pa v^j \pa v^k} \big( \frac{1}{2} F^2\big)(x, v) \bigg)$
		is positive definite at every point in $TM \backslash \{0\}$.
\end{enumerate}
Given a manifold $M$ and a Finsler structure $F$ on $TM$, the pair $(M,F)$ is called a \textit{Finsler manifold}. 	
\end{definition}

It is clear that restricting a Finsler structure $F$ to the tangent space $T_xM$ gives rise to a Minkowski space $(T_xM, F(x,\cdot))$ for each $x\in M$, and a Minkowski space can be viewed as a homogeneous Finsler manifold, where the Finsler structure is independent of the base point. 

\begin{definition}
The \textit{Legendre transform} on the Minkowski space $(\bR^n, F)$ is defined  by 
	\begin{align*}
		\mathcal{L}: I &\to (\bR^n)^{\ast} \\
		u &\mapsto p=\mathcal{L}(u)
	\end{align*}
	such that $\ker (p) = T_uI$ and $p(u) = 1$. 
\end{definition}

\subsection{Mixed norm spaces}
Given  two numbers $a>0$ and $b\ge 0$, and a positive integer $k\ge 1$,
we consider the Minkowski space $(\bR^2, F_{a,b, 2k})$,  where the norm $F_{a,b, 2k}$ is given by
\begin{equation}
	F_{a,b,2k}(v) = a\lVert v \rVert_{2} + b\lVert v \lVert_{2k}, 
\end{equation}
where $\lVert v \rVert_2 = \left( v_1^2 + v_2^2 \right)^{\frac{1}{2}}$ and $\lVert v \rVert_{2k} = \left(v_1^{2k} + v_2^{2k}\right)^{\frac{1}{2k}}$. Clearly, $F_{1,0,k}(v) = \lVert v \rVert_2$ is just the usual Euclidean norm on $\bR^2$. 

\begin{proposition}
	$(\bR^2, F_{a,b,2k})$ is a Minkowski space.  \label{prop: Mixed norm is Minkowski}
\end{proposition}

\begin{proof}
	We only need to check that condition $(iii)$ in Definition \ref{def.Finsler} is satisfied as the first two are easy to see. Let $F = F_{a,b,2k}$ and  $(g_{jk}) = \left(\left[\frac{1}{2}F^2\right]_{v_jv_k}\right)$ be the Hessian matrix of $\frac{1}{2}F^2$ whose entries are given by 	
	\begin{align*}
		g_{11} &= \left(F_{v_1}\right)^2 + FF_{v_1v_1}, \\
		g_{12} &= g_{21} = F_{v_1}F_{v_2} + FF_{v_1v_2}, \\
		g_{22} &= \left(F_{v_2}\right)^2 + FF_{v_2v_2}.
	\end{align*}
It suffices to show that $g_{11} > 0$ and $\det(g_{jk})>0$. The first and second order partials of $F$ are given by 
	\begin{align*}
		F_{v_1} &= av_1(v_1^2 + v_2^2)^{-\frac{1}{2}} + bv_1^{2k-1}(v_1^{2k} + v_2^{2k})^{\frac{1}{2k}-1}, \\
		F_{v_2} &= av_2(v_1^2 + v_2^2)^{-\frac{1}{2}} + bv_2^{2k-1}(v_1^{2k} + v_2^{2k})^{\frac{1}{2k}-1}, \\
		F_{v_1v_1} &= \alpha v_2^2 + (2k-1)\beta v_1^{2k-2}v_2^{2k}, \\
		F_{v_2v_2} &= \alpha v_1^2 + (2k-1)\beta v_1^{2k}v_2^{2k-2}, \\
		F_{v_1v_2} &= -v_1v_2\left[ \alpha + \beta(2k-1)(v_1v_2)^{2k-2} \right],
	\end{align*}
where $\alpha = a(v_1^2 + v_2^2)^{-\frac{3}{2}}$ and $\beta = b(v_1^{2k} + v_2^{2k})^{\frac{1}{2k}-2}$. Note $F_{v_1v_1},F_{v_2v_2} > 0$ implies $g_{11} > 0$. Further, $\left( F_{v_1v_2} \right)^2 = F_{v_1v_1}F_{v_2v_2}$, which we can use to show
	\begin{align*}
		\det(g_{jk}) &= F\left[(F_{v_1})^2F_{v_2v_2} - 2F_{v_1}F_{v_2}F_{v_1v_2} + F_{v_1v_1}(F_{v_2})^2 \right] \\[5pt]
		&= F\left(F_{v_1}\sqrt{F_{v_2v_2}}-F_{v_2}\sqrt{F_{v_1v_1}}\right)^2 > 0.
	\end{align*}
	This completes the proof.
\end{proof}

Starting from Section~\ref{sec.billiards.symmetry}, we will restrict to the case $k=2$, and will denote $F_{a,b} = F_{a,b, 2}$ for short.

\subsection{Curvatures in Minkowski planes}

As mentioned in the introduction, the twist coefficients of an elliptic periodic billiard orbit depend on the geometry of the billiard table, namely the curvatures at the points of reflection and the distance between them. To analyze the stability properties of billiard orbits on Minkowski planes, it is helpful to extend the concept of curvature from Euclidean spaces to  Minkowski spaces. In the following we will explain how curvatures extend to Minkowski planes  with norms following the  exposition \cite{BMS19}.

On the Euclidean plane, given a curve $\gamma$ with Euclidean length $\ell$, we can intuitively think of its curvature as how fast the tangent field varies with respect to the distance traveled on $\gamma$. Formally we define curvature as follows.
Let $\gamma:[0, \ell] \rightarrow \bR^2$ be a smooth curve with arc-length parameter $s$, $A_s$ be the area of the sector in the unit circle between $\gamma'(0)$ and $\gamma'(s)$,
and $\nu(s) = 2A_s$. The (signed) curvature of the curve $\gamma$ at $\gamma(s)$ is defined as $\kappa(s) = \nu'(s)$.
Equivalently, since $\gamma''$ is normal to $\gamma'$, we can define $\kappa(s)$ in terms of the positively oriented normal vector field $n(s)$ of unit length: $\gamma''(s) = \kappa(s)n(s)$.
Geometrically, this means that the area of the unit circle determined by the normal field $n(s)$ is the same as the area determined by the tangent field $\gamma'$. Lastly one can show that 
$\kappa(s) = \frac{1}{R(s)}$,
where $R(s)$ is the radius of the circle attached to $\gamma(s)$ such that their tangent lines agree.

While these three definitions of curvature yield the same quantity in the Euclidean plane, they diverge and produce distinct quantities in the more general Minkowski planes, as demonstrated in \cite{BMS19}. The difference arises from using a new definition of orthogonality that is not necessarily symmetric due to the lack of an inner product structure. For the rest of this section, let $(\bR^2, F)$ be a normed plane with $\lVert \cdot \rVert = F(\cdot)$ and $I$ denote the indicatrix at an arbitrary point.

\begin{definition}
	Given two vectors $v,w \in V$,  $v$ is said to be \textit{Birkhoff orthogonal} to $w$, denoted $v \dashv_B w$, if $\lVert v \rVert \leq \lVert v + tw \rVert$ for all $t \in \bR$.
\end{definition}

Geometrically, $v \dashv_B w$ indicates that $T_{v/\lVert v \rVert}I$ is parallel to $w$. Birkhoff orthogonality is not symmetric, and reversing it involves the use of a different norm. In fact, closed curves that do reverse Birkhoff orthogonality are called \textit{Radon Curves}.

\begin{definition}
	Fix a determinant form $[\cdot, \cdot]$ on a normed space $V$ i.e. a non-degenerate bi-linear form such that $[v,v] = 0$ for all $v$. The \textit{anti-norm} on $V$ is defined as 
	\begin{equation}
		\lVert v \rVert_a := \sup\{ |[w,v]| : \lVert w \rVert = 1 \} \label{eq: anti norm}
	\end{equation}
\end{definition}

In $\bR^2$, a determinant form is unique up to a constant multiple since it amounts to fixing an orientation with unit area.    Further, although $\dim V = \dim V^*$, there is no canonical isomorphism between $V$ and $V^*$. However, any isomorphism is given by an identification $v \mapsto \left[\cdot, v \right]$ for any $v \in V$, where $[\cdot, \cdot]$ is a non-degenerate bilinear form.  It follows from Eq.~\eqref{eq: anti norm} that $|[w,v]| \leq \lVert w \rVert \lVert v \rVert_a$, and it is shown in \cite{MS06} that $w \dashv_B v$ if and only if $v \dashv_B^a  w$, where $\dashv_B^a$ denotes Birkhoff orthogonality in the antinorm. Note that whenever $w \dashv_B v$ and $\Vert w \rVert = 1 $, Eq.~\eqref{eq: anti norm} is equal to the dual norm of the functional $[\cdot,v] \in V^*$. So for example, whenever $\frac{1}{p} + \frac{1}{q} = 1$, the norms on $\ell_p$ and $\ell_q$ reverse Birkhoff orthogonality.

Having redefined orthogonality we can now generalize the normal field and the first curvature concept, where curvature was given in terms of the area swept within the unit circle by the tangent field $\gamma'$. Let $n_\gamma : [0,\ell_\gamma] \rightarrow V$ be the vector field such that each $s \in [0,\ell_\gamma]$ gets mapped to the unique vector such that $\gamma'(s) \dashv_B n_\gamma(s)$ and $[\gamma'(s),n(s)] = 1$. We call $n_\gamma$ the \textit{right normal field to $\gamma$}, and in \cite{BMS19} it is shown how $n_\gamma(s)$ is unit in the anti-norm for all $s$.

\begin{definition}
	Let $\gamma : [0,\ell_\gamma] \rightarrow \bR^2$ be a smooth curve of normed length $\ell_\gamma$ with arc-length parameter $s$ (in $\lVert \cdot \rVert$), and $\varphi: [0,2A(I)] \rightarrow \bR^2$ be a parameterization of the indicatrix $I$ by twice the area of its sectors. Identify the tangent field $\gamma'$ within the indicatrix so that 	
	\begin{equation}
		\gamma'(s) = \varphi(\nu(s)),
	\end{equation}
where $\nu: [0, \ell_\gamma] \rightarrow \bR$ maps the arc-length of $T$ to twice the area of the sectors spanned by $\gamma'(0)$ and $\gamma'(s)$. The \textit{Minkowski curvature} of $\gamma$ at $\gamma(s)$ is defined as 
	\begin{equation}
		\kappa_m := \nu'(s). \label{eq: Minkowski curvature}
	\end{equation}
\end{definition}

Differentiating equation (\ref{eq: Minkowski curvature}) yields 
\begin{equation}
	\gamma''(s) = \nu'(s) \frac{d \varphi}{d \nu}(\nu(s)) = \kappa_m(s)n_\gamma(s). \label{eq: d2 gamma}
\end{equation}
Geometrically, this means that as $\gamma'$ rotates about the indicatrix $I$, $n_\gamma$ rotates about the indicatrix of the anti-norm $I_a$. The analogous definition to equation (\ref{eq: Minkowski curvature}), but where the area is spanned by $n_\gamma$ within $I_a$ is called the \textit{normal curvature} and is denoted $\kappa_n$. However, since the indicatrix in the norm and anti-norm are generally not the same, $\gamma'$ and $n_\gamma$ sweep different areas and thus yield different values for $\kappa_m(s)$ and $\kappa_n(s)$.

For later convenience, it will be useful to express equation \eqref{eq: d2 gamma} in terms of a Euclidean structure to facilitate the calculation of the higher order derivatives of $\gamma''(s)$ needed in the Taylor expansion of the billiard map $\mathcal{T}$. Given a parameterization $r(\theta(s))$ of the indicatrix $I$ such that $\gamma'(s) = r(\theta(s))\cdot (\cos\theta(s),\sin\theta(s))$ for $\theta(s) \in [0,2\pi]$, \cite{BMS19} shows how $\kappa_m$ and $n_\gamma$ may be expressed as
\begin{align}
	\kappa_m(s) &= k_e(s)r(\theta(s))^3 \label{eq: km in Euclid}\\[10pt]
	n_\gamma(s) &= \frac{1}{r(\theta(s))^2}\frac{dr}{d\theta}(\theta(s)) \cdot (\cos\theta(s),\sin\theta(s)) \nonumber \\
	&\qquad + \frac{1}{r(\theta(s))} \cdot (-\sin\theta(s),\cos\theta(s)). \label{eq: n in Euclid}
\end{align}

\subsection{Periodic orbits}\label{sec.dysy}
Let $X$ be a topological space and  $f: X \rightarrow X$ a continuous map.
Then the pair $(X, f)$ is called a dynamical system. The iterates of the map $f$ can be defined recursively via
$f^{n+1} = f\circ f^{n}$ for any $n\ge 0$. If $f$ is a homeomorphism, then we can define the backward iterates $f^{-n}$, $n\ge 1$.
A point $x \in X$ is {periodic} if $f^n(x) = x$ for some $n \ge 1$, and the {period} is the smallest $n$ satisfying this condition. Given a periodic point $x$, its orbit $\mathcal{O}(x)$ is a {periodic orbit}. If $f(x) = x$, then $x$ is a {fixed point} of $f$.

In the following we will consider a smooth surface $S$ (possibly with boundaries) with smooth area form $\omega$. 
A diffeomorphism $f: S\to S$ is said to be symplectic if it preserves the  area form $\omega$. 
Let $p$ be a fixed point of the symplectic map $f$,  $(U, \phi)$ be a local coordinate chart around $p$.
Then the tangent map $D_pf : T_p S \to T_p S$ can be viewed as a $2\times 2$ matrix with determinant $1$. Let $\lambda(p,f)$ be an eigenvalue of the matrix $D_pf$ (the other one will be $\frac{1}{\lambda(p,f)}$, if exists).
Even though the entries of $D_p f$ depend on the choice of the local coordinate system  $(U, \phi)$, 
its eigenvalues and hence the trace $\Tr(D_p f)$, do not depend on such choices. The fixed point $p$ is said to be 
\begin{enumerate}
\item parabolic if $\lambda(p,f) =\pm 1$ (or equally, $|\Tr(D_p f)| =2$);

\item hyperbolic if $|\lambda(p,f)| \neq 1$ (or equally, $|\Tr(D_p f)| > 2$);

\item elliptic if $|\lambda(p,f)|= 1$ and $\lambda(p,f)\neq \pm 1$ (or equally, $|\Tr(D_p f)| < 2$).
\end{enumerate}

\subsection{Euclidean billiards}
Let $Q \subset \bR^2$  be a bounded and connected domain with (piecewise) smooth boundary $\gamma=\pa Q$. Consider a point mass that moves within $Q$ freely and makes elastic reflects off the boundary. The induced system is called a Euclidean dynamical billiard, and $Q$ is called the billiard table. For convenience, we will assume that $Q$ is strictly convex  and $\gamma$ is smooth.

 The term Euclidean billiards is not standard. However we will use it repeatedly to distinguish it from a Minkowski billiard system that will be introduced later. All billiard systems are Euclidean for the remainder of this subsection.

The phase space $\cM$ of the dynamical billiard on the table $Q$ consists of tangent vectors
$(q,v)$, where $q\in \gamma$, and $v\in T_q Q$ is a unit vector pointing to the interior of $Q$.
Let $s$ be the arc-length parameter of $\gamma$ oriented counterclockwise, and  $\theta \in [0,\pi]$ be the angle between $v$ and $\dgs$.
Then each point $(p,v)$ can be expressed in terms of a pair of new coordinates $(s,\theta)$. Let $\ell_\gamma$ be the length of $\gamma$. Under the assumption that $Q$ is strictly convex, this gives rise to a new coordinate system on the phase space: $\mathcal{M} = \{(s,\theta): 0 \leq s \leq \ell_\gamma,\ 0 \leq \theta \leq \pi \}/\sim$ with the identification of the endpoints $s=0$ and $s=\ell_{\gamma}$. That is, the phase space $\mathcal{M}$ can be viewed as a cylinder $\bR/\ell_{\gamma} \times [0,\pi]$.

Let $(q_0,v_0) \in \mathcal{M}$ be the initial position and direction of the billiard ball. 
Let $q_1 \in \gamma$ be the first point of reflection, so that at $q_1$, the billiard changes direction from $v_0$ to $v_1$. Using their corresponding coordinates $(s_i, \theta_i)$, $i=0, 1$, we can thus define the  billiard map $T:\mathcal{M}\rightarrow \mathcal{M}$, $(s_0, \theta_0) \mapsto (s_1, \theta_1)$. 
For strictly convex billiard tables (even in $\bR^n$) whose boundary is $C^k$ for $k \geq 2$, then the billiard map $T$ is of class $C^{k-1}$. In particular, if the boundary is smooth, then $T$ is as well. By defining the billiard map, we can convert our billiard model to a discrete dynamical system on the cylinder  $\bR/\ell_{\gamma} \times [0,\pi]$.
Let $L = L(s,s_1)$ denote the distance between two points $\gamma(s)$ and $\gamma(s_1)$, and let $\kappa(s)$ and $\kappa(s_1)$ be the signed curvatures of $\gamma$ at $s$ and $s_1$ respectively.

\begin{proposition} \label{pro:DT.euclid}
The billiard map $T$ preserves the area form $\omega = \sin\theta ds \wedge d\theta$ on the phase space $\cM$.
Moreover,  the differential of the billiard map  $T: (s, \theta) \mapsto (s_1, \theta_1)$ is given by 
	\begin{equation}\label{eq: DpT}
		D_{(s,\theta)}T = \frac{1}{\sin\theta_1}
		\begin{bmatrix}
			L\kappa(s)\kappa(s_1) & L \\
			L\kappa(s)\kappa(s_1) - \kappa(s_1)\sin\theta - \kappa(s)\sin\theta_1 & \qquad L\kappa(s_1)-\sin\theta_1
		\end{bmatrix}. 
	\end{equation}
\end{proposition}
See \cite[Section 2.11]{ChMa} for more details. Our proof below is slightly different, which provides the setup that will be used again to derive the tangent map of Minkowski billiards,
see Section~\ref{sec.su} and Proposition~\ref{pro.DT.Minkowski} for more details.

\begin{proof}
	First note that $\sin\theta > 0$ so that $\omega$ is in fact an area form on $\cM$. Let $\gamma(s)$ and $\gamma(s_1)$ be distinct points on the boundary of the billiard table, and $L=L(s,s_1)$ be as above. Then $\frac{\partial L}{\partial s} = -\cos\theta$ and $\frac{\partial L}{\partial s_1} = \cos\theta_1$. It follows that 	
	\begin{align}
		dL &= \frac{\partial L}{\partial s}ds + \frac{\partial L}{\partial s_1}ds_1 = -\cos\theta ds + \cos\theta_1 ds_1, \label{eq.dLa}\\
		d^2L &= \sin\theta \ d\theta \wedge ds - \sin\theta_1 \ d\theta_1 \wedge ds_1 = 0. \nonumber
	\end{align}
This proves the first part of the theorem. 

To compute the differential $DT$, first note that since $\gamma$ is parameterized by arc-length, $\dgs \cdot \left(\gamma(s_1)-\gamma(s) \right) = L(s,s_1)\cos\theta$ and $\dgso \cdot \left(\gamma(s_1)-\gamma(s) \right) = L(s,s_1)\cos\theta_1$. Hence	
	\begin{align}
		\cos\theta &= \frac{1}{L(s,s_1)}\left(\gamma(s_1)-\gamma(s)\right)\cdot\dgs, \label{eq: costheta} \\
		\cos\theta_1 &= \frac{1}{L(s,s_1)}\left(\gamma(s_1)-\gamma(s)\right)\cdot\dgso. \label{eq: costheta1}
	\end{align}
	
Taking the differential of equation (\ref{eq: costheta}) yields 	
	\begin{align*}
		-\sin\theta d\theta &= \bigg[-\frac{1}{L^2}\frac{\partial L}{\partial s} \left(\gamma(s_1)-\gamma(s)\right) \cdot \dgs + \frac{1}{L} \bigg( \left( \gamma(s_1)-\gamma(s) \right)\cdot \ddgs -\dgs \cdot \dgs  \bigg) \bigg]ds   \\[5pt]
		&\qquad  + \bigg[ -\frac{1}{L^2}\frac{\partial L}{\partial s_1} \left(\gamma(s_1)-\gamma(s)\right) \cdot \dgs + \bigg]ds_1  \\[5pt]
		&= \left[\frac{\cos^2\theta-1}{L} + \kappa(s)\sin\theta \right]ds + \left[\frac{-\cos\theta\cos\theta_1}{L} + \frac{\cos(\theta + \theta_1)}{L}\right]ds_1  \\[5pt]		
		&= \left[\frac{-\sin^2\theta}{L} + \kappa(s)\sin\theta \right]ds - \frac{\sin\theta\theta_1}{L}ds_1. 
	\end{align*}
Solving for $\sin\theta_1ds_1$ gives
\begin{equation}
\sin\theta_1ds_1 = \left[L \kappa(s)-\sin\theta\right]ds + L d\theta. \label{eq: sintheta1}
\end{equation}
It follows that	
	\begin{align*}
		\frac{\partial s_1}{\partial s} &= \frac{L\kappa(s)-\sin\theta}{\sin\theta_1},\\[5pt]
		\frac{\partial s_1}{\partial \theta} &= \frac{L}{\sin\theta_1}.
	\end{align*}
	
Similarly, taking the differential of equation (\ref{eq: costheta1}) yields	
	\begin{align}
		-\sin\theta_1 d\theta_1 &= \left[\frac{\cos\theta\cos\theta_1}{L} - \frac{\cos(\theta+\theta_1)}{L}\right]ds + \left[\frac{1-\cos^2\theta_1}{L}-\kappa(s_1)\sin\theta_1 \right]ds_1.
	\end{align}
Simplifying the equation, we have
	\begin{align*}
		\sin\theta_1d\theta_1 = -\frac{\sin\theta\sin\theta_1}{L}ds + \left[\kappa(s_1)-\frac{\sin\theta_1}{L}\right]\sin\theta_1ds_1. 
	\end{align*}
Plugging in equation (\ref{eq: sintheta1}) into the above equation, we get 	
\begin{equation*}
		\sin\theta_1d\theta_1 = \bigg[L \kappa(s)\kappa(s_1)-\kappa(s_1)\sin\theta - \sin\theta_1\kappa(s)\bigg]ds + \bigg[\kappa(s_1)L-\sin\theta_1\bigg]d\theta.
	\end{equation*}
 It follows that	
	\begin{align*}
		\frac{\partial \theta_1}{\partial s} &= \frac{1}{\sin\theta_1}\bigg[L\kappa(s)\kappa(s_1) - \kappa(s_1)\sin\theta - \kappa(s)\sin\theta_1 \bigg],  \\[5pt]
		\frac{\partial \theta_1}{\partial \theta} &= \frac{1}{\sin\theta_1}\bigg[L\kappa(s_1)-\sin\theta_1\bigg].
	\end{align*}
Collecting terms, we complete the proof of the proposition.	
\end{proof}

Here we consider another variable $u= - \cos\theta$, which gives rise to an equivalent coordinate system of the phase space $\cM$ via $\bR/\ell_{\gamma} \times [-1, 1]$.
In particular, the corresponding area form $\omega$ on $\cM$  becomes the standard area:
\begin{align}
\omega = \sin\theta ds \wedge d\theta = ds \wedge du, \label{eq.Tsymp}
\end{align}
and Eq.~\eqref{eq.dLa} becomes 
\begin{align}
		dL &=\ -\cos\theta ds + \cos\theta_1 ds_1
		= u ds - u_1 ds_1. \label{eq.dLu}
\end{align}
In this way, the function $L=L(s,s_1)$  is a generating function of the billiard map $T: (s, u)\mapsto (s_1, u_1)$ in the sense that
\begin{align}
		u = \frac{\pa L}{\pa s}(s, s_1), \quad u_1 = -\frac{\pa L}{\pa s_1}(s, s_1). \label{eq.def.u}
\end{align}
This change is necessary when studying Minkowski billiards since the angles are no long defined.

\subsection{Stability of periodic billiard orbits} 

Now that we have introduced some basic properties of dynamical billiards, it is only natural to ask whether periodic orbits exist, and if so study their properties. Despite its simplicity, this is a surprisingly deep question. For example, in 1775 Fagnano proved that when the billiard table is an acute triangle, there always exists a periodic orbit - namely the \textit{orthic triangle}. However, the case of an obtuse triangle has proven to be much more elusive. It wasn't until the last fifteen years that \cite{Sch09} proved the existence of periodic orbits when the billiard table is a triangle whose angles are all at most one hundred degrees. 

In contrast, Birkhoff \cite{Bi27} showed that every convex billiard table with a sufficiently smooth boundary has many periodic orbits. Given a periodic orbit a natural question to ask is: What if we vary the initial conditions (maybe ever so slightly)? Do we obtain a trajectory that closely resembles the periodic orbit, or do we get something vastly different? And if the result is vastly different, is it at least in some sense 'predictable', or does it appear to be totally random? 
So the first step is to classify the periodic orbit according to the eigenvalues of the tangent map iterated up to period, see Section \ref{sec.dysy}.
It follows from the Hartman-Grobman theorem that  hyperbolic fixed points have the stable and unstable manifolds on the phase space and are locally unstable. On the other hand, the geometric characterizations for general elliptic fixed points can be very different. Given an elliptic fixed point $x^*$, the linearization  $D_{x^*}f$ acts as a rigid rotation around the fixed point. However, the local dynamics around an elliptic fixed point can be very rich for the underlying map $f$. This is in contrast to the hyperbolic case where $f$ and $D_pf$ always posses the similar dynamic behaviors. In this section we give sufficient conditions for when the elliptic islands are inherited by $f$. For more details, see the classic text \cite{SiMo}.

\begin{definition}
	A fixed point $x^*$ of  $f: U \subseteq \bR^2 \rightarrow \bR^2$ is said to be \textit{Moser stable} or \textit{nonlinearly stable} if there exists a nested sequence of $f$-invariant neighborhoods $U_n$, $n\ge 1$, containing $x^*$, whose boundaries $\partial U_n$ are invariant circles and $\bigcap\limits_{n\geq1} U_n = \{x^*\}$. 
\end{definition}

The definition of nonlinearly stable elliptic periodic points of period $n$ can be done using the iterate $f^n$.
Note that elliptic periodic points are not always nonlinearly stable and there are cases for which  elliptic periodic points are not surrounded by any invariant circle. 
As an example we consider lemon billiards first introduced in \cite{He93}. Let $Q(L)$ denote the class lemon billiard tables formed by intersecting two unit circles where $\ell$ represents the distance between the two centers. Denote by $\mathcal{O}_2(L) = \{P, \cT(P)\}$ the periodic 2-orbit reflecting between $\gamma_0(0)$ and $\gamma_1(0)$ (See Figure \ref{fig.lemon}). It is shown in \cite{JZ22} that $\mathcal{O}_2(L)$ is elliptic for $L \in (0,2)$, and nonlinearly stable for $L \in (0,2) \backslash \{1\}$. However, numerical calculations in \cite{Ch13} appear to demonstrate that the billiard map is ergodic on $Q(1)$ so that in particular, $\mathcal{O}_2(1)$ is not nonlinearly stable.

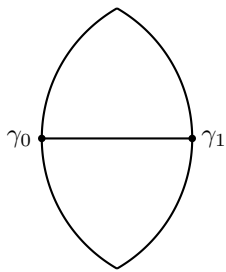
\begin{figure}[htbp]
\tikzmath{
\r=2;
\a=sqrt(3);
}
\begin{tikzpicture}
\draw[domain=-60:60, thick , samples=100] plot ({\r*cos(\x)}, {\r*sin(\x)});
\draw[domain=120:240, thick, samples=100] plot ({\r+\r*cos(\x)}, {\r*sin(\x)});
\fill (0,0) node[left]{$\gamma_{0}$} circle[radius=0.05];
\fill (2,0) node[right]{$\gamma_{1}$} circle[radius=0.05];
\draw[thick] (0,0) -- (2, 0);
\end{tikzpicture}
\caption{The lemon table $Q(1)$.}\label{fig.lemon}
\end{figure}

\subsection{Birkhoff normal form and Moser's twist mapping theorem}\label{sec.BNF}
Let $f: U\to \bR^2$ be a symplectic embedding and $P=(0,0)$ be an elliptic fixed point of $f$.
That is, the eigenvalues of the tangent matrix $D_P f$  satisfies $|\lambda| =1$ and $\lambda \neq \pm 1$.
Then the point $P$ is said to be non-resonant if $\lambda^n \neq  1$ for any $n\ge 3$.
For any $N\ge 1$, there exists an area preserving change of coordinates of the form
\begin{align} 
h_N: U \rightarrow \bR^2, \quad 
		\begin{bmatrix}
			x \\
			y
		\end{bmatrix}
		\mapsto  
		\begin{bmatrix}
			x + p_2(x,y) + \cdots + p_{2N+1}(x,y) \\
			y + q_2(x,y) + \cdots + q_{2N+1}(x,y)
		\end{bmatrix}
		+ \text{h.o.t} \label{eq: Birkhoff transformation}
\end{align}
where $p_j$ and $q_j$ are $j^\text{th}$ degree polynomials for every $j=2,3,...,2N+1$, under which one can express
\begin{equation}
		h_N^{-1}\circ f \circ h_N \left( \begin{bmatrix}
			x \\
			y
		\end{bmatrix} \right) = 
		\begin{bmatrix}
			\cos\Theta(r^2) & -\sin\Theta(r^2) \\
			\sin\Theta(r^2) & \cos\Theta(r^2)
		\end{bmatrix}
		\begin{bmatrix}
			x \\ y
		\end{bmatrix}
		+ \text{h.o.t.} \label{eq: BNF}
\end{equation}
where $r^2 = x^2 + y^2$, 
	\begin{equation}
		\Theta(r^2) = \theta + \tau_1 r^2 + \tau_2 r^4 + \cdots + \tau_N r^N, \label{eq: Theta}
	\end{equation}
and h.o.t stands for "higher order terms". \label{prop: BNF}
The symplectic transformation (\ref{eq: Birkhoff transformation}) is called the $N$-th order \textit{Birkhoff transformation}, (\ref{eq: BNF}) is called the $N$-th order \textit{Birkhoff normal form} of the map $f$ around the elliptic fixed point $P$, and 
the coefficient $\tau_j$ is called the \textit{$j$-{th} twist coefficient} (or Birkhoff coefficient) of $f$ at $P$. 
The cases such that $\lambda^n = 1$ for some $n\ge 3$ are called \textit{resonances}.
 The name "twist coefficient" is given because geometrically, $\Theta(r^2)$ defined in (\ref{eq: Theta}) measures the amount of rotation about the fixed point. 
 The same results hold for elliptic periodic orbits. See \cite{SiMo, Mos73} for more details. 
Moreover, Moeckel \cite{Moe90} studied the behavior of the first twist coefficient near an elliptic fixed point for a one-parameter family of area-preserving diffeomorphisms and gave a method for calculating $\tau_1$. These twist coefficients play an important role when determining the nonlinear stability of the elliptic fixed points.

\begin{theorem}[Moser's twist mapping theorem]\label{thm: Moser twist}
	If for an area-preserving map of the form (\ref{eq: BNF}), the polynomial $\Theta(r^2)$ defined in (\ref{eq: Theta}) is not identically zero, then the elliptic fixed point $P$ is nonlinearly stable. 
\end{theorem}

There have been many applications of Moser's Twist Mapping Theorem in the study of  dynamical billiards, see \cite{KP01, DKP03, BuGr, XiZh14} and the references therein. In \cite{KP05} the authors obtained an explicit formula of the first twist coefficients of elliptic $2$-periodic orbits in terms of the geometric characterizations of the billiard table. In \cite{JZ22} the authors studied the Birkhoff normal form around elliptic periodic points for a large class of billiards and used it to give explicit formulas for the first two twist coefficients in terms of the geometric parameters of the billiard table. The authors were then able to apply the formulas of the twist coefficients to obtain characterizations of the nonlinear stability and local analytic integrability of billiards around elliptic periodic points.

\section{Minkowski Billiards}

Minkowski billiards and Finsler billiards were introduced by Gutkin and Tabachnikov in \cite{GT02} as a natural generalization of Euclidean billiards that use a Finsler structure to represent a billiard traveling on an anisotropic and inhomogeneous medium. On Euclidean billiard systems, a billiard travels with uniform motion at every point on the interior of the billiard table -- a consequence of the indicatrix of the Euclidean metric being a unit circle at each point. However, the indicatrix of a Finsler structure is  not necessarily round or symmetric, but only  a strictly convex curve that may vary from point to point. Hence the motion of the billiard depends on its location and direction. In the remainder of this section we will restate the reflection law for Minkowski billiard systems given the lack of angles in Minkowski geometry, and prove that the Minkowski billiard map is symplectic  under an appropriate set of coordinates. 

\subsection{Reflection laws in Minkowski spaces}\label{ss.reflectionM} 
In Euclidean billiard systems, the reflection law is that the angle of incidence equals the angle of reflection. However, because Minkowski structures are not induced by an inner product, Minkowski geometry lacks the notion of angles. Hence to work with billiards in the Minkowski setting, a reflection law must be specified.  

Let $(\bR^2, F)$ be a Minkowski space,
$Q\subset \bR^2$ be a bounded and strictly convex domain with (piecewise) smooth boundary $\partial Q$. As usual, $Q$ will be called the billiard table and $\pa Q$ will be parameterized by a piece-wise smooth curve $\gamma(s)$ (counterclockwise) with $F(\dot\gamma(s)) =1$ when it exists. Further, let $y_0$ and $y_1$ be two points in the interior of $Q$ and $x$ be a point on the boundary of $Q$. Then the Minkowski distances from $y_0$ to $x$ and from $x$ to $y_1$ are given as $F(x-y_0)$ and $F(y_1-x)$ respectively. Then $xy_1$ is the reflected ray of $y_0x$ if and only if $x$ is a critical point of the distance function $F(x-y_0) + F(y_1-x)$. In \cite{GT02} Gutkin and Tabachnikov established the reflection law for Finsler billiards using the geometry of the indicatrix as follows:

\begin{proposition}[Finsler reflection law]\label{pro.Finsler.reflection.law}
	Let $y_0,y_1 \in Q$, and $x \in \pa Q$ be as above so that $xy_1$ is the billiard reflection of $y_0x$, and let $u,v \in I_x$ be the $F$-unit vectors traveling along $y_0x$ and $xy_1$ respectively. Then $\mathcal{L}(v)-\mathcal{L}(u)$ is conormal to $T_x \partial Q$ i.e. $\left(\mathcal{L}(v)-\mathcal{L}(u)\right)(w) = 0 $ for $w \in T_x \partial Q$. 
\end{proposition}

Given an incident billiard trajectory, we can use Proposition \ref{pro.Finsler.reflection.law} to find the reflected trajectory in the following manner:
\begin{enumerate}
\item $T_uI_x$ is not parallel to $T_x\partial Q$: the line $T_uI_x$ intersects  $T_x\partial Q$
at a unique point, say $w \in T_x\partial Q$. The Finsler reflection law states that the reflected vector $v \in T_xI$  satisfies $\mathcal{L}(u)(w) = \mathcal{L}(v)(w) = 1$. Therefore, so $w \in T_vI_x$. This  determines $v$, see Fig.~\ref{fig.F.reflection}.

\item $T_uI_x$ is parallel to $T_x\partial Q$:  $\mathcal{L}(u)$ is proportional to a covector that is conormal to $T_x\partial Q$. The  Finsler reflection law states that $\mathcal{L}(v)$, where $v$ is the reflected ray, is also proportional to $\mathcal{L}(u)$. Hence $T_vI_x$ is parallel to $T_uI_x$ thus determining the reflected vector $v$. 
\end{enumerate}
It is important to note that Finsler billiards are generally irreversible, i.e. given an incident and reflected billiard trajectory, we cannot reverse their roles unless the norm $F$ reversible: $F(x,-\tilde{v}) = F(x,\tilde{v})$ for $\tilde{v} \in T_xM$. It is clear that the reflection is that of equal angles when the indicatrix is a circle centered at the origin.

\begin{figure}[htbp]
\tikzmath{
\a=18;
\b=155;
}
\begin{tikzpicture}[scale=1.5]
\draw [domain=0:360, samples=100] plot ({2*cos(\x) + 1}, {0.5*sin(\x)}) node[yshift=-0.2in]{$I_x$};
\draw [thick, domain=-0.6:0.7, samples=100] plot ({\x}, {\x-\x*\x}) node[xshift=-0.6in,yshift=-0.6in]{$\pa Q$};
\draw (-1,-1) -- (2,2);
\draw ({2*cos(\a) + 1}, {0.5*sin(\a)}) -- (1.5,1.5) node[pos=0.5, above right]{$T_v I_x$} node(w){} -- ({2*cos(\b) + 1}, {0.5*sin(\b)}) node[pos=0.5, above left]{$T_u I_x$};
\draw ({2*cos(\a) + 1}, {0.5*sin(\a)}) node(v){} -- (0,0) node[pos=0.03, sloped]{$>$} node[pos=0.5](A){}; 
\draw[dashed] ({2*cos(\b) + 1}, {0.5*sin(\b)}) node(u){} -- (0,0) node[pos=0.1, sloped]{$<$} node[pos=3](B){}; 
\draw[thick] (0,0) -- (A) node[pos=0.5, sloped]{$>$};
\draw[thick] (0,0) -- (B) node[pos=0.5, sloped]{$<$};
\fill (0,0) circle(0.06) node[below]{$x$};
\fill (u) circle(0.05) node[left]{$u$};
\fill (v) circle(0.05) node[right]{$v$};
\fill (w) circle(0.05) node[left]{$w$};
\end{tikzpicture}
\caption{An illustration of the Finsler reflection law.}\label{fig.F.reflection}
\end{figure}
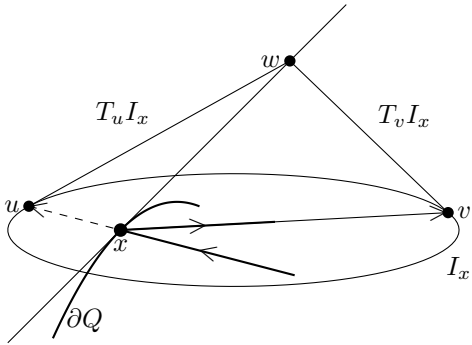

\subsection{New coordinate system on $\cM$ for Minkowski billiards}\label{sec.su}
Let $(\bR^2, F)$ be a Minkowski plane, $Q\subset \bR^2$ be a strictly convex domain with (piecewise) smooth boundary. The phase space $\mathcal{M}$ of the Minkowski billiard on $Q$ is given by all points $(x, v)$ where $x \in \partial Q$, and $v \in I_x$  points to the interior of $Q$, as we have had for Euclidean billiards. The Minkowski billiard map $\mathcal{T}$ is defined using the new reflection law in Proposition \ref{pro.Finsler.reflection.law}. That is, let $x_0 = \gamma(s_0)$ and $x_1 = \gamma(s_1)$ be points on $\partial Q$, $v_0 \in I_{x_0}$ be the $F$-unit vector at $x_0$ pointing to $x_1$, and $v_1 \in I_{x_1}$ the $F$-unit vector pointing in the direction of the reflected billiard trajectory. Then the Minkowski billiard map is $\mathcal{T}: \mathcal{M} \rightarrow \mathcal{M}$, $(s_0,v_0) \mapsto  (s_1,v_1)$.

Next we introduce a new variable $u=u(s, s_1)$ and a new coordinate system $(s,u)$ on $\cM$ for the billiard map $\cT$ using the generating function method mimicking what has been done in Eq.~\eqref{eq.def.u}. 
\begin{definition}\label{def.new.variable}
Let $Q \subset (\bR^2, F)$ be a compact and convex domain with piecewise smooth boundary $\pa Q$,
$\gamma(s)$ be a parametrization of $\pa Q$ with the $F$-arc-length parameter $s$ in the sense that $F(\dot \gamma(s)) =1$, $L(s,s_1) = F(\gamma(s_1)-\gamma(s))$  the Minkowski length from $\gamma(s)$ to $\gamma(s_1)$. Then the following defines  a new variable 
\begin{align}
u = u(s, s_1) := \frac{\partial L}{\partial s}(s,s_1). \label{def.u}
\end{align}
\end{definition}

It is easy to see that this introduces a new coordinate system $(s,u)$ on $\cM$.
Note that the lower and upper bounds of $u$ depend on the first coordinate $s$, and hence the phase space $\cM$ becomes a cylinder with varying lower and upper boundaries.
Applying the above definition to the next iterate of the billiard orbit, we get that $u_1 =u(s_1,s_2)$. Then the Minkowski billiard map can be rewritten using the new coordinate as $\mathcal{T}: (s,u) \mapsto (s_1,u_1)$. 

\begin{remark}
It is worth pointing out that the second identity in Eq.~\eqref{eq.def.u} should not be interpreted as the definition of the coordinate $u_1$. Rather, it represents a property of the Euclidean billiard map. We will verify that the same identity does hold for Minkowski billiard maps, see \eqref{eq.verify.u1} in the proof of Proposition \ref{pro.DT.Minkowski}.
\end{remark}

\subsection{The tangent map $D\cT$} 
With these new coordinates we can derive the tangent maps for Minkowski billiards as we have done in Proposition \ref{pro:DT.euclid} for Euclidean billiards. 

\begin{proposition} \label{pro.DT.Minkowski}
	The Minkowski billiard map $\mathcal{T}$ preserves the area form $\omega = ds \wedge du$. Moreover, the tangent map
	$D_{(s,u)}\mathcal{T} = 
	\begin{bmatrix}
		\frac{\partial s_1}{\partial s} & \frac{\partial s_1}{\partial u} \\
		\frac{\partial u_1}{\partial s} & \frac{\partial u_1}{\partial u},
	\end{bmatrix}$ with entires given by
	\begin{align}
		\dsds &= \frac{\sum\limits_{i,j=1}^2 \Fij\dgis\dgjs - \sum\limits_{i=1}^2\Fi\ddgis}{\sum\limits_{i,j=1}^2\Fij\dgis\dgjss}, \label{eq: dsds} \\[10pt]
		\dsdu &= \frac{-1}{\sum\limits_{i,j=1}^2\Fij\dgis\dgjss}, \label{eq: dsdu} \\[10pt]
		\duds &= \sum\limits_{i,j}^2\Fij \dgiss \dgjs - \bigg(\sum\limits_{i,j=1}^2\Fij \dgiss \dgjss \nonumber \\[10pt]
		& \qquad + \sum\limits_{i=1}^2 \Fi \ddgiss \bigg) \left( \frac{\sum\limits_{i,j=1}^2 \Fij\dgis\dgjs - \sum\limits_{i=1}^2\Fi\ddgis}{\sum\limits_{i,j=1}^2\Fij\dgis\dgjss} \right), \label{eq: duds} \\[10pt]
		\dudu &= \frac{\sum\limits_{i,j=1}^2\Fij \dgiss \dgjss + \sum\limits_{i=1}^2\Fi \ddgiss}{\sum\limits_{i,j=1}^{2}\Fij \dgis \dgjss}. \label{eq: dudu}
	\end{align}
\end{proposition}

\begin{proof}
We proceed in much the same as in Proposition \ref{pro:DT.euclid}. Let $\gamma(s),\gamma(s_1)$, and $\gamma(s_2)$ be points on the boundary $\pa Q$ such the the segment $\gamma(s_1)\gamma(s_2)$ is the reflected billiard trajectory of the billiard trajectory $\gamma(s)\gamma(s_1)$. 
Recall the distance function  $L(s,s_1) = F(\gamma(s_1)-\gamma(s))$ and the variable $u =u(s,s_1) = \frac{\partial L}{\partial s}(s,s_1)$ as above. Moreover,  $u_1 = u(s_1,s_2)$. 
By the billiard reflection law, $s_1$ is a critical point of the combined distance $L(s,s_1) + L(s_1,s_2)$. Differentiating $s_1$, we get  
\begin{align*}
\frac{\partial}{\partial s_1}\left[ L(s,s_1) + L(s_1,s_2) \right] 
=\frac{\partial L}{\partial s_1}(s,s_1)+  u(s_1,s_2)
=\frac{\partial L}{\partial s_1}(s,s_1)+ u_1= 0.
\end{align*} 
Therefore, 
\begin{align}
u_1 = -\frac{\partial L}{\partial s_1}(s,s_1). \label{eq.verify.u1}
\end{align} 
This establishes the second identity in \eqref{eq.def.u} for Minkowski billiards.
Applying the newly obtained identity \eqref{eq.verify.u1}, we can rewrite the total differential of $L(s, s_1)$ as
\begin{equation}
dL(s,s_1) = \frac{\partial L}{\partial s}(s,s_1) ds + \frac{\partial L}{\partial s_1}(s,s_1) ds_1 = u \, ds - u_1 \, ds_1. \label{eq.dL.Minkowski}
\end{equation}
Taking exterior differential of \eqref{eq.dL.Minkowski} again, we get $ddL = du \wedge ds - du_1 \wedge ds_1 = 0$, thus proving the first part of the proposition. 

Next we will derive the differential of the Minkowski billiard map $\cT: (s,u) \mapsto (s_1, u_1)$. To shorten our expressions, we will use $v = \gamma(s_1)-\gamma(s)$. Then $L(s,s_1) =F(v)$ and 
	\begin{align}
		u &= \frac{\partial L}{\partial s}(s,s_1) = - \sum\limits_{j=1}^2 F_{v_j}(v)\dgis, \\
		u_1 &= -\frac{\partial L}{\partial s_1}(s,s_1) = -\sum\limits_{j=1}^2 F_{v_j}(v)\dgiss,
	\end{align}
where we have applied \eqref{eq.verify.u1} in the second equation.
Taking the differentials of $u$ and $u_1$ yield
		\begin{align}
		du &= \left[ \sum\limits_{j,k =1}^2 F_{v_jv_k}(v)\dgis\dgjs - \sum\limits_{j=1}^2 F_{v_j}(v)\ddgis \right]ds \nonumber \\[10pt]
		&\qquad - \sum\limits_{j,k=1}^2 F_{v_jv_k}(v)\dgis\dgjss \ ds_1  \label{eq: du} \\[10pt]
		du_1 &= \sum\limits_{j=1}^2 F_{v_jv_k}(v) \dgiss \dgjs \ ds - \bigg[ \sum\limits_{j,k=1}^2 F_{v_jv_k}(v) \dgiss \dgjss \nonumber \\[10pt]
		&\qquad + \sum\limits_{j=1}^2 F_{v_j}(v) \ddgiss \bigg]ds_1. \label{eq: du1}
	\end{align}
Solving for $ds_1$ in Eq.~(\ref{eq: du}) gives
		\begin{align}
		ds_1 &= \frac{\sum\limits_{i,j=1}^2 \Fij\dgis\dgjs - \sum\limits_{i=1}^2\Fi\ddgis}{\sum\limits_{i,j=1}^2\Fij\dgis\dgjss} \ ds - \frac{du}{\sum\limits_{i,j=1}^2\Fij\dgis\dgjss} \label{eq: ds1}.
	\end{align}
Substituting Eq.~\eqref{eq: ds1} into Eq.~\eqref{eq: du1} and collecting $ds$ and $du$ terms yield
		\begin{align}
		du_1 &= \left[ \sum\limits_{i,j}^2\Fij \dgiss \dgjs - \bigg(\sum\limits_{i,j=1}^2\Fij \dgiss \dgjss \right. \nonumber \\[10pt]
		& \qquad \left. + \sum\limits_{i=1}^2 \Fi \ddgiss \bigg) \left( \frac{\sum\limits_{i,j=1}^2 \Fij\dgis\dgjs - \sum\limits_{i=1}^2\Fi\ddgis}{\sum\limits_{i,j=1}^2\Fij\dgis\dgjss} \right) \right] \ ds \nonumber \\[10pt]
		& \qquad + \frac{\sum\limits_{i,j=1}^2\Fij \dgiss \dgjss + \sum\limits_{i=1}^2\Fi \ddgiss}{\sum\limits_{i,j=1}^{2}\Fij \dgis \dgjss}du \label{eq: du1 second}.
	\end{align}
Then the entries of the tangent map $D\cT$ follows from Eq.~\eqref{eq: ds1} and Eq.~\eqref{eq: du1 second}.
\end{proof}

\section{Billiards in Mixed Norm Spaces}\label{sec.billiards.symmetry}

In this section we will study the properties of  periodic 2-orbits of Minkowski billiards on the Minkowski plane $(\bR^2,F)$, where the Minkowski norm $F(v) = F_{a,b}(v) = a\lVert v \rVert_2 + b\lVert v \rVert_4$ with $a > 0$, $b \geq 0$, and $a+b=1$. We have showed that $(\bR^2,F)$ is a Minkowski space in Proposition \ref{prop: Mixed norm is Minkowski}, and in fact is a normed space so that the indicatrix $I$ given by the level curve $F(v) = 1$ is symmetric about the origin. The assumption $a+b=1$ not only makes it so that the calculations for the twist coefficient $\tau_1$ are easier to manage, but make it so that the billiard has identical motion to a Euclidean billiard system along the horizontal and vertical axes (see Figure \ref{fig: Mixed norm indicatrix}).

\begin{figure}[htbp]
	\centering
	\subcaptionbox{$a = 0.7$ and $b = 0.3$}{\includegraphics[width=0.32\textwidth]{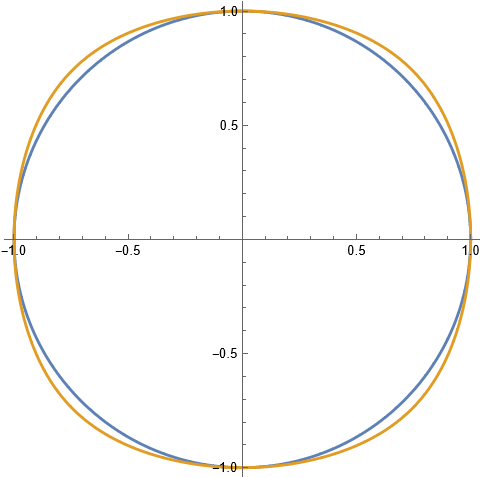}}%
	\hfill % <-- Seperation
	\subcaptionbox{$a = 0.4$ and $b = 0.6$}{\includegraphics[width=0.32\textwidth]{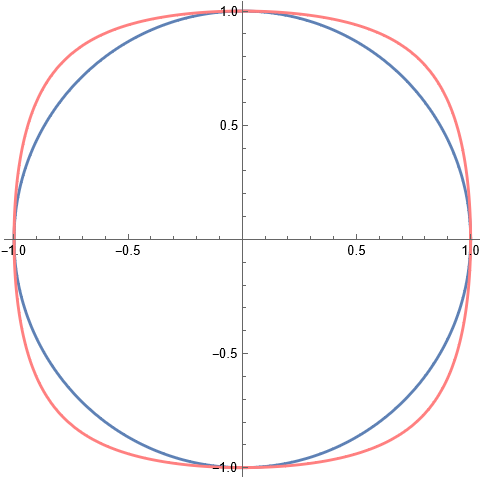}}%
	\hfill % <-- Seperation
	\subcaptionbox{$a = 0.1$ and $b = 0.9$}{\includegraphics[width=0.32\textwidth]{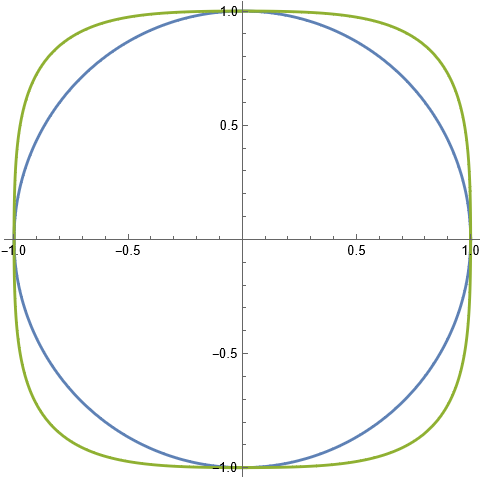}}%
	\caption{Indicatricies of Minkowski spaces with the mixed norm $F_{a,b}$.} \label{fig: Mixed norm indicatrix}
\end{figure}

Let $\alpha(t) = \sum\limits_{n \geq 1} a_{2n}t^{2n}$ and $\beta(t) = \sum\limits_{n\geq1}b_{2n}t^{2n}$, $t \in (-\epsilon, \epsilon)$, be two even functions for some $\epsilon > 0$, and $\gamma_0: t\mapsto (\alpha(t),t)$ and $\gamma_1: t\mapsto (L-\beta(t),-t)$, $t \in (-\epsilon, \epsilon)$, be two curves on $\bR^2$. The class of billiard tables $Q(L,\alpha,\beta)$ we will consider are those whose boundary $\partial Q$ is given by a piece-wise smooth curve $\gamma$ consisting of two horizontal line segments connecting the curves $\gamma_0$ and $\gamma_1$ at their endpoints (see Figure \ref{fig.btable}). Rather than parameterizing $\gamma$ with a global Minkowski arc-length parameter, we will use a local arc-length parameter $s$ on each of the curves $\gamma_j$ such that $F(\gamma_j'(s)) \equiv 1$, $s(\gamma_j(0)) = 0$ for $j = 0,1$. It is clear from the Finsler reflection law that there exists a periodic 2-orbit $\mathcal{O}_2$, bouncing back and forth between $(0,0)$ and $(L,0)$ where $L$ is the Euclidean length between the two reflection points.

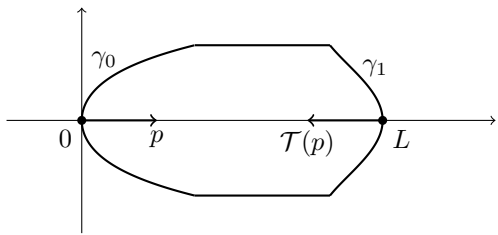
\begin{figure}[htbp]
\begin{tikzpicture}
\draw[->] (-1,0) -- (5.5,0);
\draw[->] (0,-1.5) -- (0,1.5);
\draw [domain=-1:1, samples=100, thick] plot ({4-\x*\x +0.3*\x*\x*\x*\x }, {\x});
\node at (3.9, 0.7) {$\gamma_1$};  
\draw [domain=-1:1, samples=100, thick] plot ({\x*\x +0.5*\x*\x*\x*\x }, {\x});
\node at (0.3,0.8) {$\gamma_0$};
\draw[thick, ->] (0,0) node[below left]{$0$} -- (1,0) node[below]{$p$};
\draw[thick, ->] (4,0) node[below right]{$L$} -- (3,0) node[below]{$\cT(p)$};
\draw[thick] (3.3, 1) -- (1.5, 1) (3.3, -1) -- (1.5, -1);
\fill (0,0) circle[radius=0.06];
\fill (4,0) circle[radius=0.06];
\end{tikzpicture}
\caption{An illustration of the billiard table.}
\label{fig.btable}
\end{figure}

The phase space $\mathcal{M}$ is the set of unit vectors with foot on $\partial Q$ that point inside $Q$. Recall that the new coordinates $(s,u)$ on $\mathcal{M}$ introduced in Section~\ref{sec.su}, where $u = u(s, s_1) = \frac{\partial L}{\partial s}(s,s_1)$, see Definition~\ref{def.new.variable}. Therefore the billiard map $\mathcal{T}: \mathcal{M} \rightarrow \mathcal{M}$ is given by $\mathcal{T}: (s,u) \mapsto (s_1,u_1)$, where $u_1 = u(s_1,s_2)$. For example, the periodic 2-orbit in Figure \ref{fig.btable} is given by $\mathcal{O}_2 = \{p,\mathcal{T}(p)\}$ where $p = (0,0)$ and $\mathcal{T}(p) = (0,0)$ with respect to the corresponding local coordinate systems of the phase space on the two curved parts.

\subsection{Symmetric Minkowski billiards}\label{ss.sym.billiards}

In Section~\ref{sec.BNF} we have stated how Moser's Twist Mapping Theorem (Theorem \ref{thm: Moser twist}) utilizes the twist coefficients $\tau_j$ in the Birkhoff normal form to obtain the nonlinear stability of elliptic periodic billiard orbits. We now construct Birkhoff transformations and give an explicit formula for the first twist coefficient in terms of the geometric properties of the Minkowski billiards. We will then use the first twist coefficient to investigate  the nonlinear stability property of periodic orbits of Minkowski billiards.

We will start with the case $\alpha(t) = \beta(t)$, so that the billiard table $Q(L, \alpha, \alpha)$ is symmetric about the line $x = L/2$. Combining with the symmetry of the Minkowski norm $F_{a,b}$, we get that the radius of curvature functions $R_j(s)$ on the two parts $\gamma_j$ are identical, so we set $R(s) = R_0(s) = R_1(s)$. Moreover, let $\text{Rot}_\pi$ be the rotation of the table $Q(L,\alpha,\alpha)$ about the point $(\frac{L}{2},0)$ by $\pi$. Since $F(v) = F(-v)$ then $\text{Rot}_\pi$ identifies the phase space $\mathcal{M}_1$ of $\gamma_1$ with the phase space $\mathcal{M}_0$ of $\gamma_0$.\footnote{This is false for general Minkowski norms that don't satisfy $F(-v) = F(v)$.}

\subsection{Reducing the two-step map using the symmetry}\label{ss.reducing.one.step}
Pick a  small  neighborhood $U_0 \subset \mathcal{M}_0$ of $P\in \cM_0$.
Then $U_1 = \text{Rot}_\pi(U_0) \subset \mathcal{M}_1$  is a small neighborhood of  $\mathcal{T}(P) \in \cM_1$. Let $\cT_i = \cT|_{U_i}: U_i \to \cM_{1-i}$ be the restriction of the billiard map $\cT$ to $U_i$, $i=0,1$. Observe that $ \text{Rot}_\pi\circ\cT_0 = \cT_1\circ\text{Rot}_\pi$. It follows that, restricting on $U_0$, 
\begin{align}
\cT^2 = \cT_1\circ \cT_0 = \text{Rot}_\pi\circ\cT_0 \circ \text{Rot}_\pi\circ\cT_0
=(\text{Rot}_\pi\circ\cT_0)^2.
\end{align}
If follows that the twist coefficient $\tau_1(\cT^2, P)= 2\tau_1(\text{Rot}_\pi\circ\cT_0, P)$.
Identifying $\cM_1$ with $\cM_0$ using $\text{Rot}_\pi$, we can identify the one-step billiard map $\cT_0$ with the abstract self-map $\text{Rot}_\pi\circ\cT_0: U_0\to \cM_0$. In the following we will abuse our notation and use $\cT$ for the abstract map $\text{Rot}_\pi\circ\cT_0$.
In particular, $\tau_1(\cT^2, P)= 2\tau_1(\cT, P)$.

\subsection{Tangent matrix of the billiard map}\label{ss.tangent.map}
We start with computing the tangent matrix of the billiard map $\mathcal{T}$.  The first and second order partial derivatives of the Minkowski norm $F(v)=a \|v\|_2 + b\|v\|_4$ are given by  

\begin{equation*}
	\begin{aligned}[c]
		F_{v^1} &= \frac{av_1}{\sqrt{v_1^2 + v_2^2}} + \frac{bv_1}{\left( v_1^4 + v_2^4 \right)^{\frac{3}{4}}}\\[10pt]
		F_{v^2} &= \frac{av_2}{\sqrt{v_1^2 + v_2^2}} + \frac{bv_2}{\left( v_1^4 + v_2^4 \right)^{\frac{3}{4}}}
	\end{aligned}
	\qquad\qquad
	\begin{aligned}[c]
		F_{v^1v^1} &= \frac{av_2^2}{\left(v_1^2 + v_2^2 \right)^{\frac{3}{2}}} + \frac{3bv_1^2v_2^4}{\left( v_1^4 + v_2^4 \right)^{\frac{7}{4}}} \\[10pt]
		F_{v^1v^2} &= -\frac{av_1v_2}{\left(v_1^2 + v_2^2\right)^{\frac{3}{2}}} - \frac{3bv_1^3v_2^3}{\left(v_1^4 + v_2^4 \right)} \\[10pt]
		F_{v^2v^2} &= \frac{av_1^2}{\left(v_1^2 + v_2^2 \right)^{\frac{3}{2}}} + \frac{3bv_1^4v_2^2}{\left( v_1^4 + v_2^4 \right)^{\frac{7}{4}}}.
	\end{aligned}
\end{equation*}
Note that the tangent vector at $(0,0)$ and $(L,0)$ are given by $\gamma'_0(0) = \langle 0, -1 \rangle$ and $\gamma'_1(0) = \langle 0, 1 \rangle$ respectively. Recall $\gamma''(s) = \kappa_m(s)n_\gamma(s)$ where $\kappa_m$ and $n_\gamma$ can be written in terms of a parameterization $r(\theta(s))$ of the indicatrix as shown in equations (\ref{eq: km in Euclid}) and (\ref{eq: n in Euclid}). Given the norm $F(v) = a\lVert v \rVert_2 + b\lVert v \rVert_4$ for $a>0$ and $b \geq 0$, the parameterization of $\gamma_0$ is given by
\begin{equation}
	r(\theta(s)) = \frac{1}{a + b \left(\cos^4\theta(s) + \sin^4\theta(s) \right)^{\frac{1}{4}}}.
\end{equation}
Moreover, from equation (\ref{eq: n in Euclid}) we see that $\kappa_m$ may be written in terms of the Euclidean radius of curvature since $\kappa_e(s) = \frac{1}{R(s)}$. Letting $R(0) = R$ and $a+b = 1$, Eq.~\eqref{eq: dsds}--\eqref{eq: dudu} at $P = (0,0)$ simplify to:
\begin{align}
	\dsds(0,0) &= \frac{L}{aR}-1, \label{eq: dsds mixed norm} \\[10pt] 
	\dsdu(0,0) &= \frac{L}{a}, \label{eq: dsdu mixed norm} \\[10pt]
	\duds(0,0) &= \frac{L-2aR}{aR^2}, \label{eq: duds mixed norm} \\[10pt]
	\dudu(0,0) &= \frac{L}{aR}-1. \label{eq: dudu mixed norm}
\end{align}
The tangent matrix of the one-step billiard map $\mathcal{T}$ at $P$ is thus given by
\begin{equation}
	D_p\mathcal{T} = 
	\begin{bmatrix}
		\ta_{10} & \ta_{01} \\
		\tb_{10} & \ta_{01}		
	\end{bmatrix}
	:=
	\begin{bmatrix}
		\frac{L}{aR}-1 & \frac{L}{a} \\[10pt]
		\frac{L-2aR}{aR^2} & \frac{L}{aR}-1
	\end{bmatrix}. \label{eq: DpT mixed norm}
\end{equation}
Note that in the Euclidean case where the indicatrix is a unit circle  ($a = 1$, $b = 0$),  Eq.~\eqref{eq: DpT mixed norm} reduces to the more familiar tangent matrix $D_P\mathcal{T} =
\begin{bmatrix}
	\frac{L}{R} -1 & L \\
	\frac{L}{R^2} - \frac{2}{R} & \frac{L}{R} -1
\end{bmatrix}$, see \cite[Eq.~(1.3)]{JZ22}.
Because $\mathcal{T}$ is symplectic, the eigenvalues $\lambda$ of $D_p\mathcal{T}$ satisfy $\lambda^2 - \Tr(D_P\mathcal{T})\lambda + 1 = 0$, where $\Tr(D_P\mathcal{T}) = \frac{2L}{aR} - 2$. It follows that  $D_P\cT$ is parabolic if $L = 2aR$, hyperbolic if $L > 2aR$, and elliptic if $ 0 < L < 2aR$. Going forward, we will restrict our discussion to the elliptic case. Then the eigenvalues of the tangent mattrix $D_P\cT$ are given by  $\lambda = \ta_{10} - i\sqrt{-\ta_{01}\tb_{10}}$ and $\overline{\lambda} = \frac{1}{\lambda}$, see Eq.~\eqref{eq: DpT mixed norm}. The following non-resonance condition is needed to find the Birkhoff normal form and the first twist coefficient:
\begin{itemize}
	\item [(A)] $\lambda^4 \neq 1 \text{ or equivalently, } L \in \left(0, 2aR \right) \backslash \bigl \{aR \bigl \}$. 
\end{itemize}

\begin{theorem}\label{thm.t1.sym}
Let $0< a \le 1$, $a + b=1$, $(\bR^2, F)$ be the Minkowski plane with the mixed norm $F=F_{a,b}$, $Q(L,\alpha, \beta)$ be the symmetric Minkowski billiard table,  $R_0 = R_1 = R$ be the curvature at $t=0$ (the vertices). Suppose the nonresonance condition (A) is satisfied. Then the first twist coefficient of the one-step billiard map $\mathcal{T}:(s,u) \mapsto (s_1, u_1)$ at the elliptic periodic point $P$ is given by 
	\begin{equation}
		\tau_1(\cT, P) = \frac{a^2 L R R''+(4-3 a) a L+2 (2 a-9) a R+12 R}{8 a^2 R (L-2 a R)} \label{eq: symmetric twist}
	\end{equation}
\end{theorem}

\begin{remark}
As done in \cite{JZ22}, the functions $\alpha(t)$ and $\beta(t)$ are assumed to be even polynomials in order to simplify the computation of $\tau_1$. Without these assumptions, the expression for $\tau_1$ would be significantly more complex as it would also depend on $R'$. The proof of Theorem \ref{thm.t1.sym} is given over the subsequent sections and follows closely the approaches given in \cite{JZ22, Moe90}.
Plugging in $a=1$ and $b=0$, we recover the first twist coefficient for the periodic orbit of period $2$ for Euclidean billiards:
$\ds \tau_1 = \frac{1}{8}\left(\frac{1}{R}-\frac{LR''}{2R-L}\right)$, see also \cite{KP05, JZ22}.
\end{remark}

\subsection{Taylor polynomials of the billiard map}\label{ss.Taylor.coefficients}

To begin, we first compute the Taylor expansion of the billiard map $\mathcal{T}(s,u) = (s_1,u_1)$ at $P = (0,0)$. We have computed the first order partial derivatives of $s_1$ and $u_1$ at $P$  which are given in equations \eqref{eq: dsds mixed norm}--\eqref{eq: dudu mixed norm}. To compute the higher order derivatives, we can apply the chain rule to equations \eqref{eq: dsds}--\eqref{eq: dudu}. More details for this calculation can be found in \cite{Vil}. Note that $R(s)$ has critical points at $\gamma_0(0)$ and $\gamma_1(0)$. Thus $R'(0) = 0$, and because of this, further calculations show that $\frac{\partial^{j+k} s_1}{\partial s^j \partial u^k}(0,0) = \frac{\partial^{j+k} u_1}{\partial s^j \partial u^k}(0,0) = 0$ for $j+k = 2$. Denote $R^{(j)}(0) = R^{(j)}$, we list the third order partial derivatives of $s_1$ and $u_1$ needed for the calculation of $\tau_1$:

\begin{flalign}
	\frac{\partial^3 s_1}{\partial s^3}(0,0) &= \frac{1}{a^4R^5} \bigg[2a^4LR^2-3a^3R \left(L^2-LR+2R^2\right)-a^3LR^3R'' \nonumber \\
	&\qquad + a^2 \left(L^3+3 L R^2\right)-3 a L R (L-3 R)-6 L R^2 \bigg], 
\end{flalign}
\begin{flalign}	
	\frac{\partial^3 s_1}{\partial s^2 \partial u} (0,0) &= \frac{1}{a^4R^4} \bigg[ 2a^4LR^2-2a^3R \left(L^2+R^2\right)+a^2 \left(L^3+2LR^2\right) \nonumber \\ 
	&\qquad -3aLR(L-3R)-6LR^2 \bigg],
\end{flalign}
\begin{flalign}	
	\frac{\partial^3 s_1}{\partial s \partial u^2}(0,0) &= \frac{L \left(-a^3 L R+a^2 \left(L^2+R^2\right)-3 a R (L-3 R)-6 R^2\right)}{a^4 R^3}, 
\end{flalign}
\begin{flalign}	
	\frac{\partial^3 s_1}{\partial u^3}(0,0) &= \frac{L \left(a^2 L^2-3 a R (L-3 R)-6 R^2\right)}{a^4 R^2}, 
\end{flalign}
\begin{flalign}	
	\frac{\partial^3 u_1}{\partial s^3}(0,0) &= \frac{1}{a^4R^6} \bigg[ 6LR^2-8a^5R^3-12a^3LR(L+R)+3a^2L^2(L + 3R)  + 2a^4R^2(8L + 3R)\nonumber \\
	&\qquad - 3aL(L^2 - LR + 3R^2)+ aR(L^3-3aL^2R+4a^2LR^2-2a^3R^3) R^{''} \bigg], 
\end{flalign}
\begin{flalign}	
	\frac{\partial^3 u_1}{\partial s^2 \partial u}(0,0) &= \frac{-1}{a^4R^5} \bigg[ 6a^4 LR^2-a^3R \left(9L^2+3 LR + 2R^2\right)  +a^2 L \left(3 L^2+6 L R+R^2\right) \nonumber \\
	&\qquad- 3aL\left(L^2- LR + 3R^2\right) + aLRR''(L-a R)^2 + 6LR^2 \bigg],
\end{flalign}
\begin{flalign}	
	\frac{\partial^3 u_1}{\partial s \partial u^2}(0,0) &= \frac{-1}{a^4R^4} \bigg[ 6LR^2 + 2a^4 LR^2 + a^2L(3L^2 + 3LR + 2R^2) \nonumber \\
	&\qquad - 3aL(L^2 - LR + 3R^2) - 2a^3(3L^2R + R^3) + aL^2R(L - aR)R^{''}  \bigg], 
\end{flalign}
\begin{flalign}	
	\frac{\partial^3 u_1}{\partial u^3}(0,0) &= \frac{3 (a-1) L \left(a^2 L R-a \left(L^2-L R+R^2\right)+2 R^2\right)-a L^3 R R''}{a^4 R^3}.
\end{flalign}

The Taylor polynomial of $\mathcal{T}(s,u) = (s_1, u_1)$ at $P = (0,0)$ is thus given by
\begin{align}
	s_1(s,u) &= \ta_{10}s + \ta_{01}u + \sum_{j+k = 3} \ta_{jk}s^ku^k + \text{h.o.t} \label{eq: s1 mixed norm}\\
	u_1(s,u) &= \tb_{10}s + \tb_{01}u + \sum_{j+k = 3} \tb_{jk}s^ku^k + \text{h.o.t}, \label{eq: u1 mixed norm}
\end{align}
where $\ta_{jk} = \frac{1}{j!k!} \frac{\partial^{j+k} s_1}{\partial s^j \partial u^k}(0,0)$, $\tb_{jk} = \frac{1}{j!k!} \frac{\partial^{j+k} u_1}{\partial s^j \partial u^k}(0,0)$, and h.o.t stands for higher order terms.

\subsection{The first transformation: a rigid rotation}\label{ss.rigid}
Having computed the third order Taylor polynomial of the billiard map $\mathcal{T}$, we may now initiate the construction of a symplectic coordinate transformation for $\mathcal{T}$ to be in Birkhoff normal form. The first step consists of finding a coordinate transformation for which $D_p\mathcal{T}$ acts as a rigid rotation.

Given $D_P\mathcal{T}$ as in equation (\ref{eq: DpT mixed norm}), $\lambda = \ta_{10} - i\sqrt{-\ta_{01}\tb_{10}}$ is an eigenvalue of $D_P\mathcal{T}$ with corresponding eigenvector $v_\lambda = \left[ \begin{smallmatrix} i\eta \\[7pt] \eta^{-1} \end{smallmatrix} \right]$, where $\eta = \left(-\ta_{01}\tb_{10}\right)^{1/4} > 0$. It follows that $D_P\mathcal{T}$ is conjugate to a rotation matrix $R_\theta$ where $\theta \in (0, 2\pi)$ is the argument of $\lambda$. That is, we can write $D_PT = BR_\theta B^{-1}$ where

\begin{equation}
	R_\theta = 
	\begin{bmatrix}
		\cos \theta & -\sin \theta \\
		\sin \theta & \cos \theta 
	\end{bmatrix},
	\qquad
	B = 
	\begin{bmatrix}
		\eta & \ 0 \\
		0 & \ \eta^{-1}
	\end{bmatrix}.
\end{equation}

Define the coordinate transformation $(s,u) \mapsto (x,y)$ given by 

\begin{equation}
	\begin{bmatrix}
		x \\ y
	\end{bmatrix}
	= B^{-1}\begin{bmatrix}
		s \\ u
	\end{bmatrix}
	= \begin{bmatrix}
		s/\eta \\ \eta u
	\end{bmatrix}.
\end{equation}

\noindent We abuse notation by adopting the convention that $\mathcal{T}$ will always denote the billiard map regardless of the coordinate system. This will hold true for later transformations as well. Therefore we denote $\mathcal{T}(x,y) = (x_1, y_1) = (s_1/\eta, \eta u_1)$, and it follows form equations (\ref{eq: s1 mixed norm}) and (\ref{eq: u1 mixed norm}) that

\begin{align}
	\begin{bmatrix}
		x_1 \\ y_1
	\end{bmatrix}
	&= B^{-1}
	\begin{bmatrix}
		\ta_{10}s + \ta_{01}u + \sum\limits_{j+k =3} \ta_{jk}s^ju^k \\
		\tb_{10}s + \tb_{01}u + \sum\limits_{j+k =3} \tb_{jk}s^ju^k
	\end{bmatrix} 
	+ \ \text{h.o.t} \label{eq: x1y1 B-1 mixed norm} \\[5pt]
	&= R_\theta
	\begin{bmatrix}
		x \\ y
	\end{bmatrix}
	+
	\begin{bmatrix}
		\sum\limits_{j+k = 3} a_{jk}x^jy^k \\ 
		\sum\limits_{j+k = 3} b_{jk}x^jy^k
	\end{bmatrix}
	+ \ \text{h.o.t} \label{eq: x1y1 R mixed norm}
\end{align}

\noindent where we obtained (\ref{eq: x1y1 B-1 mixed norm}) from (\ref{eq: x1y1 R mixed norm}) by using the relation 
$\begin{bmatrix}
	s \\ u
\end{bmatrix}
=
B
\begin{bmatrix}
	x \\ y
\end{bmatrix}
= 
\begin{bmatrix}
	\eta x \\ y/\eta
\end{bmatrix}$.

\noindent The coefficients $a_{jk}$ and $b_{jk}$ of the billiard map $\mathcal{T}$ in the new coordinates $(x,y)$ are given by

\begin{equation}
	a_{jk} = \eta^{j-k-1}\ta_{jk} \qquad \text{and} \qquad b_{jk} = \eta^{j-k+1}\tb_{jk}
\end{equation}

\noindent for $j+k = 1,3$. Finally, note that $\mathcal{T}:(x,y) \mapsto (x_1,y_1)$ is symplectic since both $B$ and $T:(s,u) \mapsto (s_1,u_1)$ are symplectic. 

\begin{figure}[htbp]
	\begin{align*}
		\xymatrixcolsep{3pc}
		\xymatrixrowsep{3pc}
		\xymatrix{(s,u) \ar[d]_{\mathcal{T}} \ar[r]^{B^{-1}} & (x,y) \ar[d]_{\mathcal{T}}\\
			(s_1,u_1) \ar[r]_{B^{-1}} & (x_1,y_1)} 
	\end{align*}
	\caption{The commutative diagram of the first coordinate transformation}
	\label{fig: first tower of transformations}
\end{figure}

\subsection{The second transformation: diagonalizing the linear terms}\label{ss.diagonal}

Next, we now need a second coordinate transformation in which $R_\theta$ becomes the diagonal matrix
$D = 
\begin{bmatrix}
	\lambda & \hspace{10pt} 0 \\
	0 & \hspace{10pt} \overline{\lambda}
\end{bmatrix}$.
To do so, we embed $\bR^2 \subset \bC^2$ and consider $\mathcal{T}: (x,y) \rightarrow (x_1, y_1)$ to be a complex symplectic function on a small neighborhood $(0,0)\in \bC^2$. Define new complex coordinates $z$ and $w$ as 

\begin{equation}
	z = x + iy \qquad w = x - iy \label{z,w def}
\end{equation}

\noindent and let  $\mathcal{T}(z,w) = (z_1, w_1) = (x_1+iy_1, x_1-iy_1)$. Note that under this coordinate transformation, the subspace $\bR^2$ of $\bC^2$ is mapped to $\{(z,w) \in \bC^2 : w = \overline{z} \}$ which may also be regarded as a real subspace. Observe further that $\det D_P\mathcal{T} \neq 0$. Thus by Inverse Function Theorem, there exists a neighborhood about $P$ on which $\mathcal{T}:(z,w) \mapsto (z_1,w_1)$ is a diffeomorphism so that $w_1 = \overline{z}_1$ if and only if $w = \overline{z}$.

To compute the coefficients of the billiard map $\mathcal{T}:(z,w) \mapsto (z_1,w_1)$, we plug the Taylor expansions of $x_1$ and $y_1$ given in equation \eqref{eq: x1y1 R mixed norm} into the definitions for $z_1$ and $w_1$ to get

\begin{align}
	z_1 &= x_1 + iy_1 \nonumber \\
	&= x\cos\theta - y\sin\theta + \sum_{j+k =3}a_{jk}x^jy^k + i\left( x\sin\theta + y\cos\theta + \sum_{j+k = 3} b_{jk}x^jy^k \right) + \text{h.o.t.} \nonumber \\[10pt]
	&= \lambda \left( z + \sum_{j+k=3} c_{jk}z^jw^k \right) + \text{h.o.t.}, \label{eq: z1} \\
	\intertext{and similarly,}
	w_1 &= \overline{\lambda}\left( w + \sum_{j+k = 3} \overline{c}_{jk}w^jz^k \right) + \text{h.o.t.}. \label{eq: w1}
\end{align}

\noindent The $c_{jk}$ in equations (\ref{eq: z1}) and (\ref{eq: w1}) are given by

\begin{align}
	c_{30} &= 2^{-3}\overline{\lambda}\left( a_{30} + ib_{30} - ia_{21} + b_{21} - a_{12} - ib_{12} + ia_{03} - b_{03} \right), \\
	c_{21} &= 2^{-3}\overline{\lambda}\left( 3a_{30} + 3ib_{30} - ia_{21} + b_{21} + a_{12} + ib_{12} - 3ia_{03} + 3b_{03}  \right), \\
	c_{12} &= 2^{-3}\overline{\lambda}\left( 3a_{30} + 3ib_{30} + ia_{21} - b_{21} + a_{12} + ib_{12} + 3ia_{03} - 3b_{03} \right), \\
	c_{03} &= 2^{-3}\overline{\lambda}\left( a_{30} + ib_{30} + ia_{21} - b_{21} - a_{12} - ib_{12} -ia_{03} + b_{03}\right).
\end{align}

\begin{figure}[htbp]
	\begin{align*}
		\xymatrixcolsep{3pc}
		\xymatrixrowsep{3pc}
		\xymatrix{(s,u) \ar[d]_{\mathcal{T}} \ar[r]^{B^{-1}} & (x,y) \ar[d]_{\mathcal{T}} \ar[r] & (z,w) \ar[d]_{\mathcal{T}}  \\
			(s_1,u_1) \ar[r]_{B^{-1}} & (x_1,y_1) \ar[r] & (z_1, w_1) } 
	\end{align*}
	\caption{The diagram after the second coordinate transformation}
	\label{fig: second tower of transformations}
\end{figure}
 
Since $\mathcal{T}:(x,y) \mapsto (x_1,y_1)$ is symplectic, then $D_P\mathcal{T} = 1$ in the $(x,y)$ coordinates. Further, the transformation $(x,y) \mapsto (z,w)$ has constant Jacobian. It follows that $	\frac{\partial(z_1, w_1)}{(z,w)} = 1 + (3c_{30} + \overline{c}_{12})z^2 + (2c_{21}+2\overline{c}_{21})zw + (c_{12}+3\overline{c}_{30})w^2 + \text{h.o.t.}$ is identically equal to one so that all the non-constant terms vanish. We thus get the following relations:

\begin{equation}
	c_{12} = -3\overline{c}_{30} \qquad c_{21} = \overline{c}_{21}
\end{equation}

\noindent In particular, $c_{21}$ is purely imaginary, and we can thus write $c_{21} = i \tau_1$ where $\tau_1 \in \bR$ is given by: 

\begin{equation}
	\tau_1 = \frac{a\left(a + 4b\right)L-2\left(a^2 -3ab-6b^2\right)R + a^2\left(a+b\right)^2LRR''}{8a^2R\left[\left(a+b\right)L-2aR\right]} \label{eq: twist mixed norm symmetric}.
\end{equation}
Shortly we will show that $\tau_1$ is in fact, the first twist coefficient of $\mathcal{T}$ at $P$.

\subsection{The third transformation: killing most 3rd-order terms}\label{ss.killing}

Now we need a transformation $(z,w) \mapsto (z',w')$ such that most third order terms cancel. Let this transformation be given by
\begin{align}
	z' &:= z + p_3(z,w) = z + \sum_{j+k = 3}d_{jk}z^jw'^k, \label{eq: z' mixed norm}\\
	w' &:= w + \overline{p}_3(w,z) = w + \sum_{j+k = 3}\overline{d}_{jk}w^jz^k.
\end{align}
Then $z'_1$ satisfies
\begin{align}
	z'_1 &= z_1 + \sum_{j+k = 3}d_{jk}z^j_1w^k_1= \lambda \left(z + \sum_{j+k = 3}c_{jk}z^jw^k \right) + \sum_{j+k = 3}d_{jk}\lambda^{j-k}z^jw^k + \ \text{h.o.t} \\[5pt]
	&= \lambda \left[ z' + \sum_{j+k = 3} \left( -d_{jk} + c_{jk} + d_{jk}\lambda^{j-k-1} \right)(z')^j(w')^k \right] + \ \text{h.o.t.} \label{eq: z1' mixed norm}.
\end{align}

\noindent Observe how in equation (\ref{eq: z1' mixed norm}) the $d_{21}$ terms always cancel. Therefore without loss of generality, we can set $d_{21} = 0$. By setting $-d_{jk} + c_{jk} + d_{jk}\lambda^{j-k-1} = 0$ for $(j,k) \in \{ (3,0), (1,2), (2,1) \}$ we obtain

\begin{equation}
	d_{30} = \frac{c_{30}}{1-\lambda^2} , \qquad
	d_{12} = \frac{c_{12}}{1-\overline{\lambda}^2} = -3\overline{d}_{30}, \qquad 
	d_{03} = \frac{c_{03}}{1-\overline{\lambda}^4},
\end{equation}

\noindent so that equation (\ref{eq: z1' mixed norm}) can be rewritten as $z'_1 = \lambda\big(z' + c_{21}(z')^2(w')\big) + \text{h.o.t.}$. The ellipticity of $D_P\mathcal{T}$ implies $\lambda \neq \pm 1$, ensuring that the expressions for $d_{30}$ and $d_{12}$ do not involve division by zero.  Similarly, assumption (A1) does the same for the expression of $d_{03}$.

Having $c_{21} = i\tau$ and $\lambda = e^{i\theta}$, we can further rewrite $z'_1$ if we restrict to the subspace  $\{w = \overline{z}\}$:
\begin{equation}
	z'_1 = e^{i\theta}\left(z' + i\tau |z'|^2z'\right) + \text{h.o.t.} = e^{i\left(\theta + \tau_1 |z'|^2 \right)}z' + h.o.t..
\end{equation}
This shows how the linear portion of billiard map $\mathcal{T}:(z',w') \mapsto (z'_1,w'_1)$ acts as a rotation about the elliptic fixed point. However, in the $(z',w')$ coordinates, $\mathcal{T}$ is not necessarily symplectic.

\begin{figure}[htbp]
	\begin{align*}
		\xymatrixcolsep{3pc}
		\xymatrixrowsep{3pc}
		\xymatrix{(s,u) \ar[d]_{\mathcal{T}} \ar[r]^{B^{-1}} & (x,y) \ar[d]_{\mathcal{T}} \ar[r] & (z,w) \ar[d]_{\mathcal{T}} \ar[r]^{} & (z',w') \ar[d]_{\mathcal{T}} \ar[r] & \cdots \\
			(s_1,u_1) \ar[r]_{B^{-1}} & (x_1,y_1) \ar[r] & (z_1, w_1) \ar[r] & (z'_1,w'_1)  \ar[r] & \cdots} 
	\end{align*}
	\caption{The diagram after the third coordinate transformation}
	\label{fig: tower of transformations}
\end{figure}

\subsection{The fourth transformation: a symplectic transformation}\label{ss.symplectic}

As already mentioned, the billiard map $\mathcal{T}:(z',w') \mapsto (z'_1,w'_1)$ given in the previous section is not necessarily symplectic. This is an obstacle to using Moser's twist mapping theorem to prove the nonlinear stability of periodic billiard orbits. In this section, we will use the previous transformation $(z,w) \mapsto (z',w')$ to compute a real symplectic transformation $h_N:(x,y) \mapsto (X,Y)$ such that 
\begin{equation}
	h_N^{-1} \circ T \circ h_N \left(
	\begin{bmatrix}
		x \\ y
	\end{bmatrix}
	\right) = 
	\begin{bmatrix}
		\cos\Theta(r) & -\sin\Theta(r) \\
		\sin\Theta(r) & \cos\Theta(r)
	\end{bmatrix}
	\begin{bmatrix}
		x \\ y
	\end{bmatrix} + \text{h.o.t.},
\end{equation}
where $r^2 = x^2 + y^2$ and $\Theta(r) = \theta + \tau_1r^2 + \tau_2r^4 + \cdots + \tau_{N-1}r^{2(N-1)}$. 

First, notice that $p_3(z,w)$ satisfies $\frac{\partial p_3}{\partial z} = - \frac{\partial \overline{p}_3}{\partial w}$. Hence there exists a polynomial $s_4(z,w)$ satisfying $\frac{\partial s_4}{\partial w}(z,w) = p_3(z,w)$ and $\frac{\partial s_4}{\partial z}(z,w) = -\overline{p}_3(w,z)$, given by 
\begin{equation}
	s_4(s,w) = \frac{1}{4}\frac{\overline{c}_{03}}{\lambda^4-1}z^4 - \frac{1}{4}\frac{c_{03}}{\overline{\lambda}^4-1}w^4 + \frac{\overline{c}_{30}}{\overline{\lambda}^2-1}zw^3 - \frac{c_{30}}{\lambda^2-1}z^3w,
\end{equation}
that is purely imaginary whenever $w = \overline{z}$. Hence we can thus define a real valued function $g_4(x+iy,x-iy) = \frac{i}{2}s_4(x+iy,x-iy)$. Next, consider the generating function $G(x,Y) := xY + g_4(x+iY, x-iY)$ of the symplectic transformation $(x,y) \mapsto (X,Y)$ where $X = X(x,y)$ and $Y = Y(x,y)$. Then 
\begin{alignat}{3}
	X &:= \frac{\partial G}{\partial Y} & &= x + \frac{i}{2} \big[ -\overline{p}_3(x-iY,x+iY)(i) + p_3(x+iY,x-iY)(-i) \big] \nonumber \\
	& & &= x + \Re\left(p_3(x+iY, x-iY)\right) \nonumber \\
	& & &= x + \Re\left( \sum\limits_{j+k=3}d_{jk}(x+iY)^j(x-iY)^k \right) \nonumber \\
	& & &=: x + \sum_{j+k =3}p_{jk}x^jY^k ,\label{eq: X}\\
	y &:= \frac{\partial G}{\partial x} & &= Y + \frac{i}{2} \bigg[ -\overline{p}_3(x-iY,x+iY) + p_3(x+iY,x-iY) \bigg] \nonumber \\
	& & &= Y - \Im\left(p_3(x+iY,x-iY)\right) \nonumber \\
	& & &= Y - \Im\left( \sum\limits_{j+k=3}d_{jk}(x+iY)^j(x-iY)^k \right) \nonumber\\
	& & &=: Y - \sum_{j+k = 3}q_{jk}x^jY^k. \label{eq: y}
\end{alignat}
It follows from equations (\ref{eq: X}) and (\ref{eq: y}), that on a small neighborhood about $(x,y) = (0,0)$, we have 
\begin{align}
	X &= x + \sum_{j+k = 3}p_{jk}x^jy^k + \text{h.o.t.},\\
	Y &= y + \sum_{j+k = 3}q_{jk}x^jy^k + \text{h.o.t.}.
\end{align}
Letting, $Z = X + iY$ and $W = X - iY$ gives
\begin{align}
	Z &= x + iy + \Re\left(p_3(x+iy, x-iy)\right) + i\Im\left(p_3(x+iy, x-iy)\right) + \text{h.o.t.} \nonumber \\
	&= z + p_3(z,w) + \text{h.o.t.}
	\intertext{Similarly,}
	W &= w + \overline{p}_3(w,z) + \text{h.o.t.},
\end{align}
so that the billiard map $\mathcal{T}: (Z,W) \mapsto (Z_1,W_1)$ is given by 
\begin{align}
	Z_1 &= e^{i\left(\theta + \tau_1|Z| \right)}Z + \text{h.o.t.}, \label{eq: Z1}\\
	W_1 &= e^{-i\left(\theta + \tau_1|Z|\right)}W + \text{h.o.t.}. \label{eq: W1}
\end{align}

Letting $\Theta = \theta + \tau_1|Z|^2$ and converting back to the variables $X$ and $Y$ yields the desired Birkhoff normal form, showing $\tau_1$ is the first twist coefficient of $\mathcal{T}$:
\begin{equation}
	\begin{bmatrix}
		X_1 \\ Y_1
	\end{bmatrix} =
	\begin{bmatrix}
		\cos\Theta & -\sin\Theta \\
		\sin\Theta & \cos\Theta 
	\end{bmatrix}
	\begin{bmatrix}
		X \\ Y
	\end{bmatrix} + \text{h.o.t.}
\end{equation}

\begin{figure}[htbp]
	\begin{align*}
		\xymatrixcolsep{3pc}
		\xymatrixrowsep{3pc}
		\xymatrix{(s,u) \ar[d]_{\mathcal{T}} \ar[r]^{B^{-1}} & (x,y) \ar[d]_{\mathcal{T}} \ar[r] & (X,Y) \ar[d]_{\mathcal{T}} \ar[r]^{} & (Z,W) \ar[d]_{\mathcal{T}}\\
			(s_1,u_1) \ar[r]_{B^{-1}} & (x_1,y_1) \ar[r] & (X_1, Y_1) \ar[r] & (Z_1,W_1)} 
	\end{align*}
	\caption{The diagram after the fourth coordinate transformation.}
	\label{fig: fourth transformation}
\end{figure}

\subsection{The first twist coefficient for asymmetric Minkowski billiards}\label{ss.asymmetry}

In this subsection we study the general case where the (Euclidean) radius of curvature functions $R_j(s)$, have different values around $s=0$. In particular, $R_0(0) = R_0$ and $R_1(0) = R_1$ may be different. For the class of billiard tables $Q(L,\alpha, \beta)$, this will occur whenever the even polynomials $\alpha(t) = \sum_{n\geq1}\alpha_{2n}t^{2n}$ and $\beta(t) = \sum_{n\geq1}\beta_{2n}t^{2n}$ have different values for $\alpha_2$ and $\beta_2$. Going forward we will assume that $\alpha_2 > \beta_2$ so that $R_0 \leq R_1$. As before, we will consider the periodic 2-orbit $\mathcal{O}_2 = \{P, T(P)\}$.

The billiard map ${\mathcal{T}}^2$ is given by the composition $(s,u) \xrightarrow{{\mathcal{T}}} (s_1,u_1) \xrightarrow{{\mathcal{T}}} (s_2, u_2)$ where as before $s_1 = s_1(s,u)$, $u_1 = u_1(s,u)$, $ s_2=s_2(s_1,u_1)$ and  $u_2 =u_2(s_1,u_1)$. The partial derivatives $\frac{\partial^{j+k}s_2}{\partial s^j \partial u^k}$ and $\frac{\partial^{j+k}u_2}{\partial s^j \partial u^k}$ are straightforward albeit tedious applications of the chain rule. To illustrate it, we give expressions for $\frac{\partial s_2}{\partial s}$ and $\frac{\partial^2 s_2}{\partial s^2}$: 
\begin{align}
	\frac{\partial s_2}{\partial s} &= \frac{\partial s_2}{\partial s_1}\dsds + \frac{\partial s_2}{\partial u_1}\duds, \\[10pt]
	\frac{\partial^2 s_2}{\partial s^2} &= \frac{\partial^2 s_2}{\partial s_1^2} \left(\dsds \right)^2 + 2 \frac{\partial^2 s_2}{\partial s_1 \partial u_1}\duds\dsds + \frac{\partial s_2}{\partial s_1}\frac{\partial^2 s_1}{\partial s^2} + \frac{\partial^2 s_2}{\partial u_1^2} \left(\duds \right)^2 + \frac{\partial s_2}{\partial u_1}\frac{\partial^2 u_1}{\partial s^2}.
\end{align}

Denote $\tilde{a}_{jk} =\frac{1}{j!k!}\frac{\partial^{j+k}s_2}{\partial s^j \pa u^k}(0,0)$ and $\tilde{b}_{jk} =\frac{1}{j!k!}\frac{\partial ^{j+k}u_2}{\partial s^j \pa u^k}(0,0)$. Calculating the remaining first order partial derivatives and evaluating at $\mathcal{O}_2$, show that $D_P{\mathcal{T}}^2$ is given by 

\begin{equation}
	D_P{\mathcal{T}}^2 =
	\begin{bmatrix}
		\ta_{10} & \ta_{01} \\
		\tb_{10} & \tb_{01}
	\end{bmatrix}
	=
	\begin{bmatrix}
		\frac{L}{aR_1}-1 & \frac{L}{a} \\
		\frac{L-a(R_0 + R_1)}{aR_0R_1} & \frac{L}{aR_0}-1
	\end{bmatrix}
	\begin{bmatrix}
		\frac{L}{aR_0}-1 & \frac{L}{a} \\
		\frac{L-a(R_0+R_1)}{aR_0R_1} & \frac{L}{aR_1}-1
	\end{bmatrix},
\end{equation}
where 
\begin{align}
	\ta_{10} &= \tb_{01} = \frac{a R_0 \left(a R_1-2 L\right)+2 L \left(L-a R_1\right)}{a^2 R_0 R_1}, \\[10pt]
	\ta_{01} &= \frac{2L}{a}\left( \frac{L}{aR_1}-1 \right), \\[10pt]
	\tb_{10 }&=\frac{2 \left(L-a R_0\right) \left(L-a \left(R_0+R_1\right)\right)}{a^2 R_0^2 R_1}.
\end{align}
Then the orbit $\mathcal{O}_2$ is parabolic for $L \in \left\{ aR_0, aR_1, a(R_0 + R_1) \right\}$, hyperbolic for $L \in \left(aR_0, aR_1\right) \cup \left(a(R_0 + R_1), \infty \right)$, and elliptic for $L \in \left(0, a(R_0 +R_1)\right) \cup \left(aR_1, a(R_0 +R_1)\right)$. We assume the following non-resonance condition:

\begin{itemize}
\item[(B)]	$\lambda^4 \neq 1$, or equivalently,  
$\left(L-aR_0\right)\left(L-aR_1 \right) \neq 0$.
\end{itemize}

\begin{theorem}\label{thm.t1.asym}
	Let $F(v) = a(v_1^2 + v_2^2)^{1/2} + b(v_1^4 + b_2^4)^{1/4}$ with $a+b=1$, and $Q(\alpha(t), \beta(t),L)$ be given. Additionally, assume that the resonance condition (B) is satisfied. Then the first twist coefficient of the two-step billiard map $T^2$ is given by
	\begin{equation}
		\tau_1 = \frac{\Delta}{8aR_0R_1(L-aR_0)(L-aR_1)\left(L-a(R_0+R_1)\right)},
	\end{equation}
	\noindent where 
	\begin{align}
		\Delta &= 2 a R_0^2 \left(-L R_1 \left(a^2 L R_1''+4 a^2-18 a+12\right)+a \left(2 a^2-9 a+6\right) R_1^2+a (3 a-4) L^2\right) \nonumber \\
		&+ 2 a R_0^2 \left(-L R_1 \left(a^2 L R_1''+4 a^2-18 a+12\right)+a \left(2 a^2-9 a+6\right) R_1^2+a (3 a-4) L^2\right) \nonumber \\
		&+ R_0 \bigg(L^2 R_1 \left(a^2 L R_0''+a^2 L R_1''+8 a^2-36 a+24\right)-2 a L R_1^2 \left(a^2 L R_0''+4 a^2-18 a+12\right) \\
		&+a^2 R_1^3 \left(a^2 L R_0''+2 a^2-9 a+6\right)+a (4-3 a) L^3 \bigg)  +a^2 R_0^3 \big(R_1 \left(a^2 L R_1''+2 a^2-9 a+6\right) \nonumber \\
		&+a (4-3 a) L\big)-a (3 a-4) L R_1 \left(L-a R_1\right){}^2. \nonumber 
	\end{align}
\end{theorem}

\begin{proof}
We only give a sketch of the computation. The detailed computation be found in \cite{Vil}. The Taylor coefficients  $\ta_{jk}$ and $\tb_{jk}$ can be done by taking derivatives of the two functions $s_2$ and $u_2$ repeatedly. Combining with the assumption on two functions $\alpha$ and $\beta$ being even, we have that   $\ta_{jk}$'s and $\tb_{jk}$'s for $j+k = 2$ are identically zero. The remaining of the computation for $\tau_1$ is exactly the same as in Section~\ref{ss.rigid}, \ref{ss.killing} and \ref{ss.symplectic}.
\end{proof}

\begin{remark}
	The first twist coefficient of elliptic periodic orbits of period $2$ in Euclidean billiards \cite{JZ22, KP05} is recovered by setting $a=1$ and $b=0$:
	\begin{equation}
		\tau_1 = \frac{1}{8}\bigg(\frac{R_0+R_1}{R_0R_1}-\frac{L}{R_0+R_1-L}\left(\frac{R_1-L}{R_0-L}R_0'' + \frac{R_0-L}{R_1-L}R_1''\right)\bigg).
	\end{equation}
\end{remark}

\section{Applications}\label{sec.applications}

In this section we use the first twist coefficient given in Theorem \ref{eq: symmetric twist} to prove the nonlinear stability of the periodic orbits of period $2$ on various symmetric billiard tables. Using the assumption $a+b =1$, we will shorten the notation for Minkowski norm as $F_{a}= F_{a, 1-a}$ with $ 0< a \le 1$, and will denote the Minkowski billiard map as $\cT_a$ to emphasize its dependence on the parameter $a$.
We have showed in
Section~\ref{ss.tangent.map} that the periodic point $P$ of period $2$ is elliptic for $0<L<2aR$, parabolic for $L=2aR$, and hyperbolic for $L>2aR$.

\subsection{Euclidean circular tables} 
Consider the Euclidean circular billiard table $Q_r$ on the Minkowski plane $(\bR^2, F_{a})$ bounded by the circle $x^2 + y^2 = r^2$. Then the orbit bouncing back and forth along the $y$-axis is a periodic orbit of period $2$,  say $\mathcal{O}_2(a, r) =\{P_r, \cT_a(P_r)\}$. In this case, we have $R = r$ and $L=2r$. From the stability conditions above, it follows that the periodic orbit $\mathcal{O}_2(a, r)$ is hyperbolic for $0< a < 1$ (see Figures \ref{fig:circleorbits}(a) and \ref{fig:circlephase}) and parabolic only for $a=1$, which corresponds to the Euclidean norm $F_{1}(v) = \|v\|_2$. In particular, such an orbit cannot be elliptic. Additionally, numerical simulations demonstrate the nonlinear stability of the periodic 2-orbits $\{(\frac{\sqrt{2}}{2},\frac{\sqrt{2}}{2}),(-\frac{\sqrt{2}}{2},-\frac{\sqrt{2}}{2})\}$ and $\{(-\frac{\sqrt{2}}{2},\frac{\sqrt{2}}{2}),(\frac{\sqrt{2}}{2},-\frac{\sqrt{2}}{2})\}$ (see Figures \ref{fig:circleorbits}(b) and \ref{fig:circlephase}). Although the twist coefficient in equation (\ref{eq: symmetric twist}) does not pertain to these periodic 2-orbits, we can obtain their twist coefficient in a similar way as presented above. However, equations (\ref{eq: dsds})-(\ref{eq: dudu}) would not simplify as nicely as they do in equations (\ref{eq: dsds mixed norm})-(\ref{eq: dudu mixed norm}).

\begin{figure}[htbp]
	\centering
	\subcaptionbox{100 reflections of the billiard orbit with initial point $(0,-1)$ and direction $\pi/2+0.01$}{\includegraphics[width=0.48\textwidth]{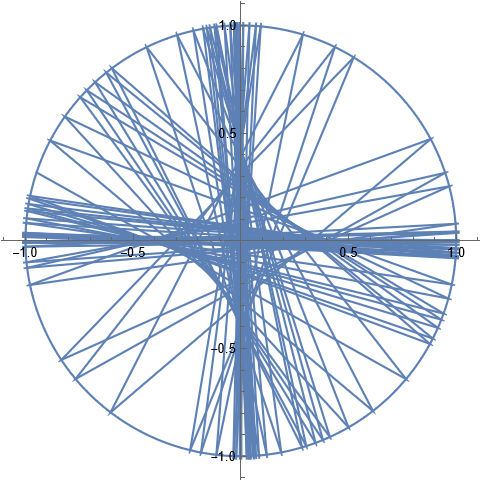}}%
	\hfill % <-- Seperation
	\subcaptionbox{100 reflections of the billiard orbit wit initial point $(\frac{\sqrt{2}}{2},\frac{\sqrt{2}}{2})$ and initial direction $\frac{5\pi}{4}+0.1$}{\includegraphics[width=0.48\textwidth]{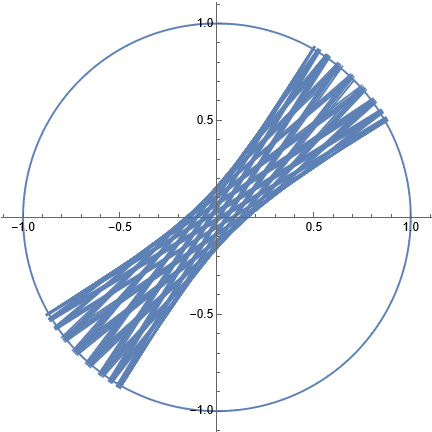}}%
	\caption{Dynamics billiards on the table bounded by $x^2 + y^2 = 1$ with Minkowski norm $F_{0.8}$.} \label{fig:circleorbits}
\end{figure}

\begin{figure}[htbp]
	\centering
	\includegraphics[width=0.7\textwidth]{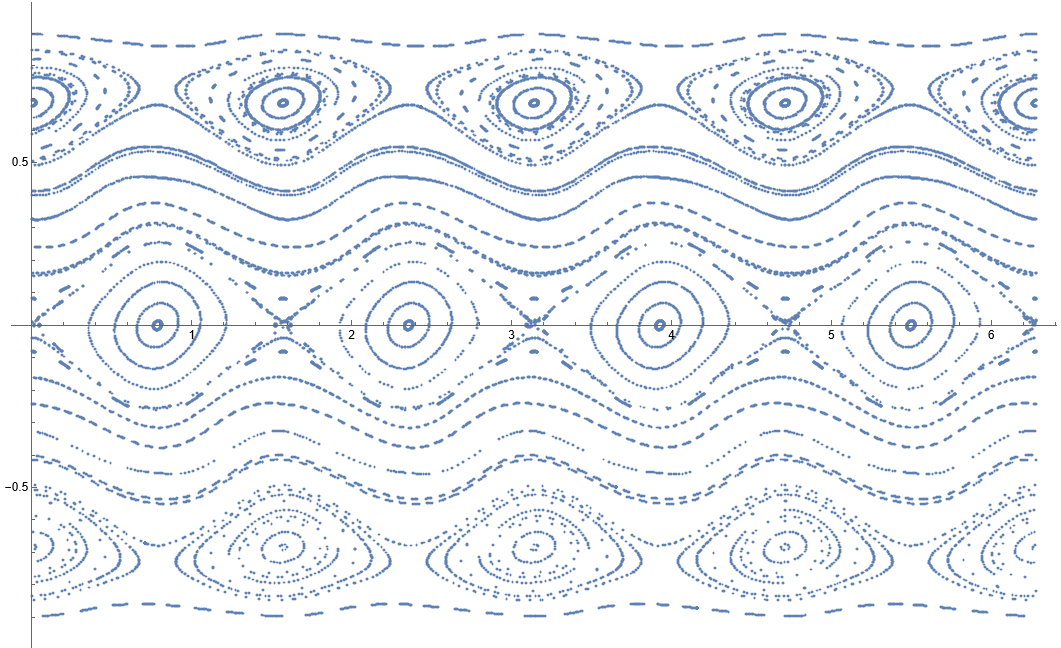}
	\caption{Phase space of Minkowski billiards with norm $F_{0.8}$ on the table bounded by $x^2 + y^2 =1$} \label{fig:circlephase}
\end{figure}

\subsection{Symmetric lemon billiards}
Next, consider the symmetric billiard table $Q(L)$ discussed in Section~\ref{ss.sym.billiards}, where $\gamma_0$ and $\gamma_1$ are Euclidean circular arcs of radius $1$. In this case, $R=1, R'' = 0$, and $L\in (0,2)$. From the stability conditions, the orbit $\mathcal{O}_2(a,L) =\{P_L, \cT_a(P_L)\}$ is elliptic whenever $L<2a$ with nonresonance condition $L\neq a$. The first twist coefficient is given by
\begin{equation}
	\tau_1(\mathcal{T}_a, P_L) = \frac{a^2 (3 L-4)+a (18-4 L)-12}{8 a^2 (2 a-L)}.
\end{equation}
Note that $\tau_1(\mathcal{T}_a, P_L)=0$ when $L = \frac{2\left(2a^2-9a+6\right)}{a\left(3a-4\right)}$.

\begin{figure}[htbp]
	\centering
	\includegraphics[width=0.8\linewidth]{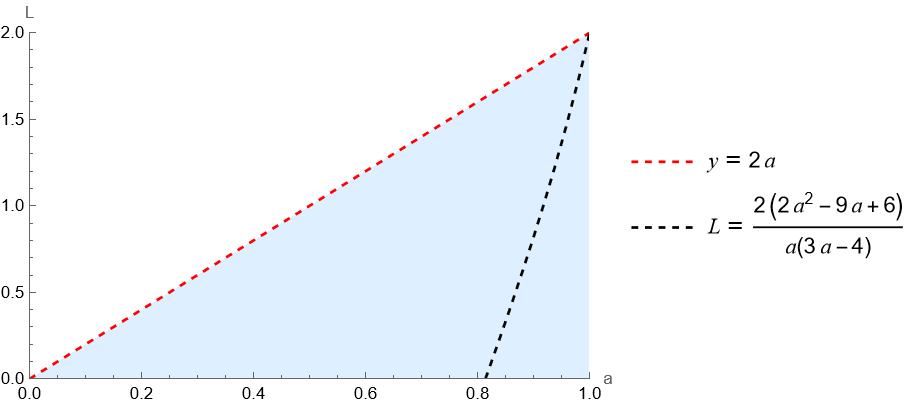}
	\caption{Stability regions for the orbit $\mathcal{O}_2(a, L)$ on a symmetric lemon table $Q(L)$ on the Minkowski plane with norm $F_{a}$.}
	\label{fig:lemonsymmetricgraph}
\end{figure}

\begin{corollary}
Let $0<a <1$
and  $\mathcal{O}_2(a, L)$ be the periodic orbit on the symmetric lemon table $Q(L)$ running along the $x$-axis. Then the orbit  $\mathcal{O}_2(a, L)$ is nonlinearly stable for  
$L \in (0,2a) \backslash \left\{ a, \frac{4a^2-18a+12}{a(3a-4)} \right\}$.
\end{corollary}
\begin{remark}
It is not clear if the periodic orbit  $\mathcal{O}_2(a, L)$ is nonlinearly stable when $L= \frac{4a^2-18a+12}{a(3a-4)}$. In this case, $\tau_1(\mathcal{T}_a, P_L) =0$, and finding the next coefficient $\tau_2(\mathcal{T}_a, P_L)$ may be helpful.
\end{remark}

\subsection{Euclidean elliptical tables}
Let $\delta \in (0,1)$, and $E(\delta)$ be the billiard table $E(\delta)$ bounded by the Euclidean ellipse given $x^2  + \frac{y^2}{\delta^2}= 1$. It is easy to see that there is a periodic orbit  $\mathcal{O}_2(a, \delta)=\{P_{\delta}, \cT_a(P_{\delta}\}$ of period $1$ running along the  minor axis of the ellipse $E(\delta)$. Then $L = 2\delta$ and $R = 1/\delta$, so that $\frac{L}{R}=2\delta^2$. It follows that $\mathcal{O}_2(a, \delta)$ is elliptic for $\delta < \sqrt{a}$, parabolic for $\delta = \sqrt{a}$, and hyperbolic for $\delta > \sqrt{a}$. Consequently, $\mathcal{O}_2(a, \delta)$ is always elliptic in the Euclidean case, and the non-resonance condition is satisfied for $\delta \neq \frac{\sqrt{a}}{2}$. A short calculation shows $R''=\frac{3\left(\delta^2-1\right)}{\delta}$ so that the first twist coefficient is given by
\begin{equation}
	\tau_1(P_{\delta}, \cT_a) = \frac{\delta\left(-6+a(9+a-4\delta^2)\right)}{8a^2(a-\delta^2)}.
\end{equation}

\begin{figure}[htbp]
	\centering
	\includegraphics[width=0.8\linewidth]{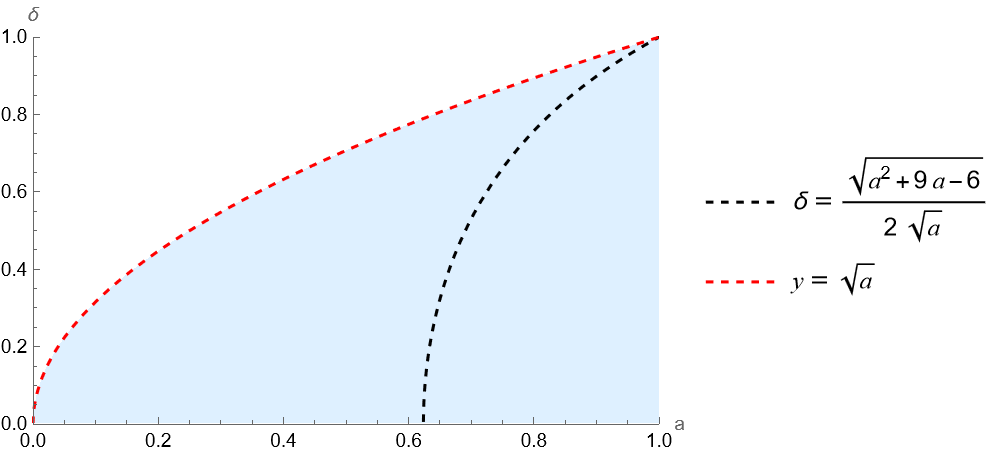}
	\caption{Stability regions for the orbit $\mathcal{O}_2(a, \delta)$ on an elliptic table $E(\delta)$ on the Minkowski plane with norm $F_{a}$.}
	\label{fig: elliptic billiard graph}
\end{figure}

Note that $\tau_1(P_{\delta}, \cT_a)=0$ when $\delta = \frac{\sqrt{a^2+9a-6}}{2\sqrt{a}}$. This value of $\delta$ only makes sense when it is real and less than $\sqrt{a}$ which occurs whenever $a \in \left(\frac{\sqrt{105}-9}{2},1\right)$

\begin{corollary}
	On the elliptical billiard table $E(\delta)$, the periodic orbit $\mathcal{O}_2(a, \delta)$ running along the minor axis is nonlinearly stable whenever $\delta < \sqrt{a}$ and $\delta \neq \frac{\sqrt{a^2 + 9a -6}}{2\sqrt{a}}$ and $\delta \neq \frac{\sqrt{a}}{2}$.
\end{corollary}

\begin{figure}[htbp]
	\centering
	\includegraphics[width=0.7\textwidth]{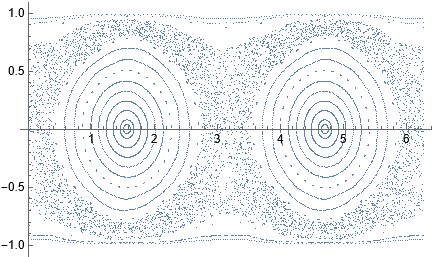}
	\caption{Phase space of Minkowski billiards with norm $F_{0.8}$ on the table $E(0.5)$}
\end{figure}

\subsection{General symmetric billiards}
For the class of billiard tables $Q(\alpha, \beta, L)$ where $\alpha(t) = \beta(t)$, $R(s)$ satisfies $R = \frac{1}{2 \alpha_2}$ and $R'' = 6\alpha_2 - \frac{6\alpha_4}{\alpha_2^2}$ at both $\gamma_0(0)$ and $\gamma_1(0)$. Hence the orbit $\mathcal{O}_2(a, \alpha)=\{P_{\alpha}, \cT_a(P_{\alpha}\}$ is elliptic whenever $L \in \left(0, \frac{a}{\alpha_2}\right)$. Further, the non-resonance condition is satisfied whenever $L \neq \frac{a}{2\alpha_2}$. From equation (\ref{eq: symmetric twist}), we have that $\tau_1(P_{\alpha}, \cT_a) =0$ when 
\begin{align}
R'' = -\frac{-3 a^2 L+4 a^2 R+4 a L-18 a R+12 R}{a^2 L R}, 
\end{align}
or equivalently, 
\begin{align}
\alpha_4 = -\frac{-2 a^2 \alpha _2^2+9 a \alpha _2^2-4 a \alpha _2^3 L-6 \alpha _2^2}{3 a^2 L}.
\end{align}

%%%%%%%%%%%%%%%%%%%%%%%%%%%%%%

\end{document}